\documentclass[12pt]{article}

\usepackage{amsfonts,mathrsfs}
\usepackage{amssymb,amsmath,amsbsy,amsthm}
\usepackage{dsfont}
\usepackage{blkarray}
\usepackage{tikz}
\usepackage{tkz-euclide}
\usepackage{tikz-cd}
\usepackage{enumerate}
\usetikzlibrary{patterns}
\usepackage{ae}
\usepackage{bm}
\usepackage{graphicx}
\usepackage{float}
\usepackage{cite}
\usepackage[d]{esvect}   
\usepackage{mathtools} 
\usepackage{indentfirst}
\usepackage{lineno}
\usepackage{titlesec}
\usepackage{enumitem}
\usepackage{color}
\usepackage{aliascnt}
\usepackage{varioref}
\usepackage{hyperref}
\hypersetup{colorlinks=true,linkcolor=cyan,anchorcolor=cyan,citecolor=red,CJKbookmarks=True,
	,urlcolor=cyan}
\usepackage{cleveref}

\numberwithin{equation}{section}

\def\And{\mbox{\rm ~and~}}

\def\For{\mbox{\rm ~for~}}

\def\I{\mbox{\rm (\hspace{0.2mm}I\hspace{0.2mm})}\,}

\DeclareFontEncoding{LS1}{}{}
\DeclareFontSubstitution{LS1}{stix}{m}{n}
\DeclareSymbolFont{stixletters}{LS1}{stix}{m}{it}
\DeclareMathAccent{\cev}{\mathord}{stixletters}{"91}
\DeclareMathAccent{\vec}{\mathord}{stixletters}{"92}
\DeclareMathAccent{\vecev}{\mathord}{stixletters}{"95}

\def\({\mbox{\rm (}}\def\){\mbox{\rm )}}

\makeatletter

\newcommand{\Rmnum}[1]{\expandafter\@slowromancap\romannumeral #1@}
\makeatother
\newtheorem{theorem}{Theorem}[section]
\newaliascnt{lemma}{theorem}
\newtheorem{lemma}[lemma]{Lemma}
\aliascntresetthe{lemma}

\newaliascnt{proposition}{theorem}
\newtheorem{proposition}[proposition]{Proposition}
\aliascntresetthe{proposition}

\newaliascnt{fact}{theorem}

\aliascntresetthe{fact}

\newaliascnt{definition}{theorem}

\aliascntresetthe{definition}

\newaliascnt{conjecture}{theorem}

\aliascntresetthe{conjecture}

\newaliascnt{corollary}{theorem}
\newtheorem{corollary}[corollary]{Corollary}
\aliascntresetthe{corollary}

\newaliascnt{claim}{theorem}

\aliascntresetthe{claim}

\newaliascnt{problem}{theorem}

\aliascntresetthe{problem}

\newaliascnt{question}{theorem}

\aliascntresetthe{question}

\newaliascnt{remark}{theorem}

\aliascntresetthe{remark}

\newaliascnt{example}{theorem}
\newtheorem{example}[example]{Example}
\aliascntresetthe{example}

\newaliascnt{notation}{theorem}

\aliascntresetthe{notation}

\begin{document}
	
\begin{center}
{\Large\bf Partition-selected flow polynomials and associated arrangements}\\ [7pt]
	\end{center}
	\vskip 3mm
\begin{center}
Beifang Chen$^{1}$ \quad Ying Cao$^{2}$\quad Houshan Fu$^{3}$ \quad{\rm and}\quad Hongyang Wang$^4$\\[8pt]
$^{1,2,4}$Department of Mathematics, Hong Kong University of Science and Technology\\
Clear Water Bay, Hong Kong, P. R. China\\[12pt]
$^{3}$School of Mathematics and Information Science, Guangzhou University\\
Guangzhou 510006, Guangdong, P. R. China\\[15pt]
E-mails: $^{1}$mabfchen@ust.hk, $^{2}$ycaobf@connect.ust.hk, $^{3}$fuhoushan@gzhu.edu.cn, $^4$hwanghj@connect.ust.hk\\[15pt]
\end{center}
	
\begin{abstract}
We introduce a partition-selection method to generalize the flow, chromatic, and Tutte polynomials of a graph by restricting the standard edge subgraph expansions to subgraphs given by prescribed connected vertex partitions. We establish similar deletion-contraction formulas and specialization relations for these polynomials, recovering all classical polynomial invariants when the selection is the set of all partitions.

Next we study a relation between  Jaeger et al.'s nonhomogeneous flows and a special class of partition-selected flow polynomials (called affine flow polynomials). Specifically, we give a geometric realization of nowhere-zero nonhomogeneous flows by restricting the edge-coordinate arrangement to affine flow spaces. The resulting characteristic polynomials coincide with Kochol's admissible assigning polynomials and with affine flow polynomials, which enumerate nowhere-zero nonhomogeneous flows over finite fields.

To see the key role of the partition-selection framework, we further introduce boundary arrangements determined by the bond structure of a graph. Using the intersection posets of boundary arrangements, we obtain the classification of all restricted arrangements mentioned above, the comparison of unsigned coefficients of affine flow polynomials, and the decomposition formulas for affine flow polynomials.
\vspace{1ex}
 \vspace{1ex}
\end{abstract}
\section{Introduction}\label{Sec-1}
Throughout this paper, let $G =\big (V(G), E(G)\big)$ be a finite graph,  possibly with loops and multiple edges. We denote by $r$ and $n$ the {\em rank function } and  the {\em nullity function} of  $G$ respectively, defined by: 
\[
r(S):=|V(G)|-c(S)\;\quad{\rm and}\;\quad n(S):=|S|-r(S) \;\For\; S\subseteq E(G),
\]
where $c(S)$ is the number of  components of the spanning subgraph $(V(G),S)$. One important graph polynomial invariant is the flow polynomial $\varphi_G(t)$ of $G$, introduced by Tutte \cite{Tutte1954}, whose value $\varphi_G(k)$ at each positive integer $k$ is the number of nowhere-zero flows over an additive abelian group of order $k$. In particular, $\varphi_G(t)$ is trivially zero if $G$ contains a cut-edge (bridge or isthmus). The {\em flow polynomial} admits the expansion over edge subsets:
\[
\varphi_G(t)=\sum_{S\subseteq E(G)}(-1)^{|S^c|}t^{n(S)},
\]
where $S^c:=E(G)\smallsetminus S$. Moreover, it also satisfies the deletion-contraction recurrence:
\[
\varphi_G(t)=
\begin{cases}
(t-1)\varphi_{G\setminus e}(t),&\mbox{if } e \mbox{ is a loop};\\
\varphi_{G/e}(t)-\varphi_{G\setminus e}(t),&\mbox{otherwise}.
\end{cases}
\]

Dual to the flow polynomial is the chromatic polynomial $\chi_G(t)$, introduced by Birkhoff \cite{Birkhoff1912} for planar graphs to address the four-color problem, and later generalized to arbitrary graphs by Whitney \cite{Whitney1932,Whitney1932-1}. Its value $\chi_G(k)$ at each positive integer $k$ is the number of proper vertex colorings using $k$ colors, where a proper vertex coloring means that ends of each edge receive distinct colors. If $G$ is oriented and colors are members of an additive abelian group, then each vertex coloring is a group-valued function and its difference produces a tension of $G$. Viewed in this way proper vertex colorings correspond to nowhere-zero tensions. The number of nowhere-zero tensions over an additive abelian group of order $k$ is given by the polynomial $\tau_G(k)$, introduced by Tutte \cite{Tutte1954} but formally named the tension polynomial by Kochol \cite{Kochol2002}. 

As generalizations of the chromatic, tension, and flow polynomials, bivariate polynomials were introduced by Whitney and Tutte. The rank generating function of $G$ has a standard edge subgraph expansion in \cite{Whitney1932-1}. Tutte \cite{Tutte1954,Tutte1967} introduced the Tutte polynomial of $G$ as a sum over all spanning forests, where the exponents correspond to internal and external activities. In fact, the Tutte polynomial also admits a standard edge subgraph expansion. All of these polynomials satisfy the fundamental deletion-contraction relations.

Inspired by the earlier work, one of our goals is to construct graph polynomials arising from certain partitions of vertices that enjoy properties analogous to those of the chromatic, tension, flow and Tutte polynomials. Specifically, the desired polynomials are required to satisfy counting interpretations, expansion formulas over edge subgraphs, and deletion-contraction recurrences. To be more precise, we first review some basic notation for graphs.

Fix an edge $e\in E(G)$. The graph $G\setminus e$  is obtained by {\em deleting} the edge $e$ from $G$, and the graph $G/e$ is obtained by {\em contracting} $e$, that is, identifying the ends of $e$ to create a single new vertex and then removing $e$. Generally, for any $S\subseteq E(G)$, we write $G\setminus S$ (respectively $G/S$) to denote the graph obtained by deleting (respectively contracting) all of the edges in $S$. In particular, $G\setminus S^c$ is the {\em spanning subgraph} $(V(G), S)$ induced by $S$, which is also denoted by $G|S$. For each vertex subset $V'\subseteq V(G)$, let $G[V']$ denote the {\em vertex-induced subgraph} of $G$, consisting of the vertex set $V'$ and all edges with ends in $V'$. In addition, the edge set $E(G)$ can be naturally decomposed into the two disjoint parts:
\[
E(G) = E_{\rm cyc}(G)\sqcup E_{\rm cut}(G),
\]
where $E_{\rm cyc}(G)$ denotes the set of {\em cycle-edges} (edges contained in at least one cycle) and $E_{\rm cut}(G)$ is the set of {\em cut-edges} (edges not contained in any cycle). 

A {\em partition} $\pi$ of $V(G)$ is a collection of nonempty, pairwise disjoint vertex subsets such that their union equals $V(G)$. Let $\Pi(G)$ denote the poset of partitions $\pi$ of $V(G)$ such that each block $V_i$ of $\pi$ induces a connected subgraph $G[V_i]$. The partial order $\pi_1\le\pi_2$ ($\pi_1$ is {\em finer} than $\pi_2$) of partitions $\pi_1$ and $\pi_2$ means that each block of $\pi_1$ is contained in a block of $\pi_2$. For any $S\subseteq E(G)$, let $\pi(S)$ be the partition of $V(G)$ induced by the spanning subgraph $G|S$, whose blocks are vertex sets of components of  $G|S$. Then $\pi$ leads to a surjective map from the set of all edge subsets of $E(G)$ to $\Pi(G)$. 
Thus, we have
\[
\Pi(G)=\big\{\pi(S):S\subseteq E(G)\big\}.
\]
For instance, the partition $\pi(\emptyset)$ is the minimal member of $\Pi(G)$ consisting of the singletons $\{v\}$ with $v\in V(G)$, and $\pi(E(G))$ is the maximal member of $\Pi(G)$ consisting of the vertex sets of components of $G$. For the empty graph $E_n$, the set $\Pi(E_n)$ contains only one partition $\pi(E_n)$ whose blocks are singletons. A subset $\mathfrak{p}$ of $\Pi(G)$ is called a {\em selection} of $G$.

Associated with a selection $\mathfrak{p}$ of $G$, we define the {\em partition-selected chromatic polynomial} $\chi_G(\mathfrak{p},t)$,  {\em partition-selected tension polynomial} $\tau_G(\mathfrak{p},t)$, {\em partition-selected flow polynomial} $\varphi_G(\mathfrak{p},t)$ and {\em partition-selected rank generating function} $R_G(\mathfrak{p},x,y)$ as:
\begin{align}
&\chi_G(\mathfrak{p},t):=\sum_{S\subseteq E(G),\,\pi(S)\in\mathfrak{p}}(-1)^{|S|}t^{c(S)},\label{Partition-Chromatic-Def}\\
&\tau_G(\mathfrak{p},t):=\sum_{S\subseteq E(G),\,\pi(S)\in\mathfrak{p}}(-1)^{|S|}t^{r(G)-r(S)},\label{Partition-Tension-Def}\\
&\varphi_G(\mathfrak{p},t):=\sum_{S\subseteq E(G),\,\pi(S)\in\mathfrak{p}}(-1)^{|S^c|}t^{n(S)},\label{Partition-Flow-Def}\\
&R_G(\mathfrak{p};x,y):=\sum_{S\subseteq E(G),\,\pi(S)\in\mathfrak{p}}x^{r(G)-r(S)}y^{n(S)}\label{Partition-Rank-Def}.
\end{align}
For the empty selection $\emptyset$ of $\Pi(G)$, we adopt the following  conventions:
\[
\chi_G(\emptyset,t)=\tau_G(\emptyset,t)=\varphi_G(\emptyset,t)=R_G(\emptyset;x,y)=0.
\]
We further define the {\em partition-selected Tutte polynomial} $T_G(\mathfrak{p};x,y)$ as 
\begin{equation}\label{Partition-Tutte-Def}
T_G(\mathfrak{p};x,y):=R_G(\mathfrak{p};x-1,y-1).
\end{equation}
When $\mathfrak{p}=\Pi(G)$, the partition-selected graph polynomial reduces to the corresponding ordinary graph polynomial. Particularly, for the empty graph $E_n$ with $n$ vertices, 
\[
\chi_{E_n}\big(\Pi(E_n),t\big)=t^n \;\And\; \tau_{E_n}\big(\Pi(E_n),t\big)=\varphi_{E_n}\big(\Pi(E_n),t\big)=T_{E_n}\big(\Pi(E_n);x,y\big)=1.
\]

The partition-selected Tutte polynomial  generalizes the partition-selected chromatic, tension and flow polynomials, see \autoref{Partition-Relations-CTF-Tutte}. \autoref{Partition-CTFT-DC} shows that all partition-selected graph polynomials satisfy the corresponding deletion-contraction recurrence relations. 

Our further purpose is to give a combinatorial interpretation for the partition-selected flow polynomial, analogous to that for the classical flow polynomial. Nowhere-zero $\mathbb{Z}_k$-flows were initially introduced by Tutte in \cite{Tutte1949,Tutte1954} as a dual concept to graph coloring. A planar graph is $k$-colorable if and only if its dual graph admits a nowhere-zero $\mathbb{Z}_k$-flow. The analogue of a $\mathbb{Z}_k$-flow is an integer $k$-flow, where the flow values on edges are integers with absolute value strictly less than $k$. A fundamental result due to Tutte \cite{Tutte1949} states that a graph has a nowhere-zero $k$-flow if and only if it admits a nowhere-zero $A$-flow for any abelian group $A$ of order $k$. Comprehensive surveys on nowhere-zero flows can be found in \cite{Jaeger1988,Seymour1995}. 

Let $A$ be a finite additive abelian group. We denote by $A^{E(G)}$ and $A^{V(G)}$ the sets of functions from $E(G)$ to $A$ and from $V(G)$ to $A$, respectively. Associated with an orientation of $G$, the {\em incidence matrix} of $G$ is the $|V(G)|\times |E(G)|$ integral matrix $M_G:=(m_{ve})$ whose rows and columns are indexed by the vertices and edges, where, for a vertex $v$ and an edge $e$,
\[
m_{ve}:=\begin{cases}
1, & \text{ if } e \text{ is a link and } v \text{ is the head of } e;\\
-1,& \text{ if } e \text{ is a link and } v \text{ is the tail of } e;\\
0,& \text{ otherwise},
\end{cases}
\]
and further let $E^+(v)$ be the set of edges with $v$ as the head, and $E^-(v)$ be the set of edges with $v$ as the tail. The {\em boundary operator}  is a map $\partial:A^{E(G)}\to A^{V(G)}$, defined by
\[
\partial\bm c(v)=\sum_{e \in E^{+}(v)} \bm c(e)-\sum_{e \in E^{-}(v)}\bm c(e).
\]
Equivalently, the boundary operator can be expressed via the incidence matrix as
\begin{equation}\label{Boundary-Matrix}
\partial\bm c(v):=\sum_{e\in E(G)}m_{ve}\bm c(e).
\end{equation}
The members in the kernel of $\partial$ are called {\em flows}. The {\em boundary group} $B(G,A)$ of $G$ is the image of $\partial$, whose members are called {\em boundary chains} of $G$.  Following \eqref{Boundary-Matrix}, we have 
\[
B(G,A)=\big\{M_G\bm c\in A^{V(G)}:\bm c\in A^{E(G)}\big\}.
\]
When $A=\mathbb{F}$ is a field, the boundary group is a vector space, called the {\em boundary space}.

To address the 3-flow conjecture (see Unsolved Problem 48, in \cite{Bondy1976}), Jaeger, Linial, Payan and Tarsi \cite{Jaeger1992} introduced the group connectivity of graphs in 1992 by generalizing the nowhere-zero flow to a nonhomogeneous form. Fix $\bm b\in B(G,A)$, and let $F(G,\bm b;A)$ denote the set of all functions $\bm c\in A^{E(G)}$ satisfying $\partial\bm c=\bm b$. Alternatively, from \eqref{Boundary-Matrix}, we have
\[
F(G,\bm b;A)=\big\{\bm c\in A^{E(G)}:M_G\bm c=\bm b\big\}.
\]
We call the elements of $F(G,\bm b;A)$ {\em affine flows} (also known as {\em  $(A,\bm b)$-flows} in \cite{Lai2006}).  Furthermore, an affine flow $\bm c$ of $G$ is said to be  {\em nowhere-zero} if $\bm c(e)\ne 0$ for all $e\in E(G)$ . The set of nowhere-zero affine flows of $G$ with boundary $\bm b$ is denoted by $F_{\rm nz}(G,\bm b;A)$.

Recently, Kochol introduced assigning polynomials to count nowhere-zero affine flows in \cite{Kochol2022}, and later extended the approach to regular matroids in \cite{Kochol2024}. Subsequently, Fu, Ren and Wang  provided an explicit expression for assigning polynomials in \cite[Theorem 1.3]{FRW2025}, and then investigated the  properties of their coefficients. More specifically, let $\Lambda(G)$ be the family of nonempty vertex subsets $X\subseteq V(G)$ such that $G[X]$ is connected and $c(G[V(G)\smallsetminus X])=c(G)$. Adopting Kochol's notation, a {\em $\{0,1\}$-assigning} of $G$ is a map $\alpha$ from $\Lambda(G)$ to the set $\{0,1\}$. Each function  $\bm b\in A^{V(G)}$ automatically induces a $\{0,1\}$-assigning  $\alpha_{G,\bm b}$ on $\Lambda(G)$ satisfying that for each $X\in\Lambda(G)$,  $\alpha_{G,\bm b}(X)=0$ if $\sum_{v\in X}\bm b(v)=0$, and $\alpha_{G,\bm b}(X)=1$ otherwise. For any $\bm b\in B(G,A)$, the {\em assigning polynomial} $\varphi_G(\alpha,t)$ $(\alpha=\alpha_{G,\bm b})$ is given by
\begin{equation}\label{AP}
\varphi_G(\alpha,t):=\sum_{S\subseteq E(G),\,G\setminus S\text{ is $\bm b$-compatible}}(-1)^{|S|}t^{n(S^c)},
\end{equation}
where $G\setminus S$ is {\em $\bm b$-compatible} if $\sum_{v\in H}\bm b(v)=0$ for each component $H$ of $G\setminus S$. 

In this paper, we shall restrict our attention to nowhere-zero affine flows over fields. Our final focus is to study the combinatorial aspects of nowhere-zero affine flows using restricted arrangements and boundary arrangements, and then to find their connections to the partition-selected flow polynomials and assigning polynomials. 

The paper is organized as follows. \autoref{Sec2-0} presents the main results, and the remaining sections contain their proofs.
\section{Main results}\label{Sec2-0}
This section states the main results, and their proofs will be given in later sections. Throughout this paper,  for each edge $e$ of $G$ with ends $u$ and $v$ (not necessarily distinct), we always set $G':=G\setminus e$ and $G^{''}:=G/e$. Let $\Pi_e(G)$ be the set of partitions $\pi\in\Pi(G)$ such that the ends of $e$ are contained in a block of $\pi$, i.e., 
\begin{equation}\label{Pi-e}
\Pi_e(G):=\big\{\pi\in\Pi(G):\{u,v\} \text{ is contained in a block of } \pi\big\}.
\end{equation}
Then, a selection $\mathfrak{p}$ of $G$ induces a selection $\mathfrak{p}'$ of  $G'$ and a selection $\mathfrak{p}^{''}$ of $G^{''}$,  defined by
\begin{equation}\label{Triple-Partitions}
\mathfrak{p}':=\mathfrak{p}\cap\Pi(G')\quad\And\quad \mathfrak{p}^{''}:=\big\{\pi/e:\pi\in\mathfrak{p}\cap\Pi_e(G)\big\}\subseteq\Pi(G^{''}),
\end{equation}
where $\pi$ is required not to separate ends of $e$, and $\pi/e$ is the partition of $G^{''}$ obtained from $\pi$ by identifying the ends of $e$ into a single new vertex of $G^{''}$. Our first main result is as follows.
\begin{theorem}\label{Partition-CTFT-DC}Let $\mathfrak{p}\subseteq\Pi(G)$ and $e\in E(G)$. The partition-selected chromatic, tension, flow and  Tutte polynomials  satisfy
\begin{align}
&\chi_G(\mathfrak{p},t)=\begin{cases} 0,&\text{ if } e \text{ is a loop};\\
\chi_{G'}(\mathfrak{p}',t)-\chi_{G^{''}}(\mathfrak{p}^{''},t),&\text{ otherwise},
\end{cases}\label{Partition-Chromatic-DC}\\
&\tau_G(\mathfrak{p},t)=\begin{cases} 0,&\text{ if } e \text{ is a loop};\\
t\tau_{G'}(\mathfrak{p}',t)-\tau_{G^{''}}(\mathfrak{p}^{''},t),&\text{ if } e \text{ is a cut-edge};\\
\tau_{G'}(\mathfrak{p}',t)-\tau_{G^{''}}(\mathfrak{p}^{''},t),&\text{ otherwise},
\end{cases}\label{Partition-Tension-DC}\\
&\varphi_G(\mathfrak{p},t)=\begin{cases}
 (t-1)\varphi_{G'}(\mathfrak{p}',t),&\text{ if } e \text{ is a loop};\\
\varphi_{G^{''}}(\mathfrak{p}^{''},t)-\varphi_{G'}(\mathfrak{p}',t),&\text{ otherwise},
\end{cases}\label{Partition-Flow-DC}\\
&T_G(\mathfrak{p};x,y)=\begin{cases}
yT_{G'}(\mathfrak{p}';x,y),&\text{ if } e \text{ is a loop};\\
(x-1)T_{G'}(\mathfrak{p}';x,y)+T_{G^{''}}(\mathfrak{p}^{''};x,y),&\text{ if } e \text{ is a cut-edge};\\
T_{G'}(\mathfrak{p}';x,y)+T_{G^{''}}(\mathfrak{p}^{''};x,y),&\text{ otherwise}.
\end{cases}\label{Partition-Tutte-DC}
\end{align}
Moreover, the partition-selected chromatic and tension polynomials are related by 
\[
\chi_G(\mathfrak{p},t)=t^{c(G)}\tau_G(\mathfrak{p},t).
\]
\end{theorem}

Next, we give a geometric realization of nowhere-zero affine flows using certain restricted arrangements. To this end, let us review basic definitions of hyperplane arrangements (see \cite{Stanley2007}). A {\em hyperplane arrangement} $\mathcal{A}$ is a finite set of (affine) hyperplanes in a vector space $V$ over a field $\mathbb{F}$. The {\em intersection poset} $L(\mathcal{A})$ is a poset consisting of all nonempty intersections of some  elements in $\mathcal{A}$, ordered by the reverse inclusion, and including the whole space $V=\bigcap_{H\in\emptyset}H$ as the minimal element. Every member in $L(\mathcal{A})$ is called a {\em flat}. As a generalization of the chromatic polynomial, the {\em characteristic polynomial} $\chi(\mathcal{A},t)$ of $\mathcal{A}$ is 
\[
\chi(\mathcal{A},t):=\sum_{\mathcal{B}\subseteq \mathcal{A},\,\bigcap_{H\in\mathcal{B}}H\ne\emptyset}(-1)^{|\mathcal{B}|}t^{{\rm dim}\bigcap_{H\in\mathcal{B}}H}.
\]
Fix a flat $X\in L(\mathcal{A})$. The {\em restriction} $\mathcal{A}/X$ of $\mathcal{A}$ to $X$ is a hyperplane arrangement in $X$, given by
\[
\mathcal{A}/X=\{H\cap X\ne\emptyset: H\in \mathcal{A}\And X\nsubseteq H\}.
\]
The {\em complement} of $\mathcal{A}$ is  $M(\mathcal{A}):=V\smallsetminus \bigcup_{H\in \mathcal{A}}H$. The ambient space $V$ has the set partition
\begin{equation}\label{DA}
V=\bigsqcup_{X\in L(\mathcal{A})}M(\mathcal{A}/X),
\end{equation}
i.e., for each point $\bm p\in V$ there is a unique flat $X$ of $\mathcal{A}$ satisfying $\bm p\in M(\mathcal{A}/X)$, denoted by $X_{\bm p}$. Indeed, $X_{\bm p}$ is the inclusion-minimal element of $L(\mathcal{A})$ containing $\bm p$. The property implies that the minimal flat $X_{\bm p}$ containing $\bm p$ can be expressed as
\begin{equation}\label{Flat}
X_{\bm p}=\bigcap_{H\in\mathcal{A},\, \bm p\in H}H.
\end{equation}

Associated with a finite graph $G$, the {\em coordinate arrangement} of $G$ is a hyperplane arrangement $\mathcal{A}_{E(G)}$ in the vector space $\mathbb{F}^{E(G)}$, defined by
\[
\mathcal{A}_{E(G)}:=\big\{H_e:x_e=0\mid e\in E(G)\big\}.
\]
Recall that an affine flow $\bm c$ with boundary $\bm b$ is nowhere-zero if and only if $\bm c(e)\ne 0$ for all $e\in E(G)$. Given a function $\bm b\in \mathbb{F}^{V(G)}$, we naturally consider two types of {\em restricted arrangements} $\mathcal{A}_{E(G)}^{\bm b}$  and $\tilde{\mathcal{A}}_{E(G)}^{\bm b}$ of $\mathcal{A}_{E(G)}$ by restricting $\mathcal{A}_{E(G)}$ to the affine subspace $F(G,\bm b;\mathbb{F})$. Specifically, the restriction of $\mathcal{A}_{E(G)}$ to $F(G,\bm b;\mathbb{F})$ yields the subspace arrangement in $F(G,\bm b;\mathbb{F})$, which may contain the whole space $F(G,\bm b;\mathbb{F})$ and is given by 
\[
\mathcal{A}_{E(G)}^{\bm b}:=\big\{H_e^{\bm b}:=H_e\cap F(G,\bm b;\mathbb{F}): H_e^{\bm b}\ne\emptyset, e\in E(G)\big\}.
\]
Similarly, we denote by $\tilde{\mathcal{A}}_{E(G)}^{\bm b}$ the hyperplane arrangement in $F(G,\bm b;\mathbb{F})$ obtained from $\mathcal{A}_{E(G)}^{\bm b}$ by removing the ambient space $F(G,\bm b;\mathbb{F})$, i.e., $\tilde{\mathcal{A}}_{E(G)}^{\bm b}:=\mathcal{A}_{E(G)}^{\bm b}\smallsetminus F(G,\bm b;\mathbb{F})$. Notably, when $\bm b\in\mathbb{F}^{V(G)}\smallsetminus B(G,\mathbb{F})$, we have $F(G,\bm b;F)=\emptyset$. In this case, we naturally have $\mathcal{A}_{E(G)}^{\bm b}=\tilde{\mathcal{A}}_{E(G)}^{\bm b}=\emptyset$, and the corresponding characteristic polynomial is the zero polynomial. Therefore, we will restrict our attention to $\bm b\in B(G,\mathbb{F})$ in the remainder of this paper.

For each partition $\pi\in \Pi(G)$, a {\em boundary subspace} $B_\pi$ is defined by
\[
B_\pi:=\Big\{\bm x\in \mathbb{F}^{V(G)}: \sum_{v\in V_i}\bm x(v)=0\text{ for every block }V_i\in\pi\Big\}.
\]
In fact, $B_\pi$ is a subspace of the boundary space $B(G,\mathbb{F})$, and also coincides with the boundary space of some spanning subgraph of $G$, i.e.,
\begin{equation}\label{Identity}
B(G\setminus S,\mathbb{F})=B_{\pi(S^c)},
\end{equation}
whose detailed explanations are presented in \autoref{Sec4-1}. For any $S\subseteq E(G)$, set $H_S^{\bm b}:=\bigcap_{e\in S}H_e^{\bm b}$. Then $H_S^{\bm b}$ consists of affine flows with boundary $\bm b$ that vanish on every edge in $S$, and hence it is naturally identified with $F(G\setminus S,\bm b;\mathbb{F})$. Combining this with
\eqref{Identity}, we obtain
\begin{equation}\label{eq:equivalence for coset intersection nonempty}
H_S^{\bm b}\ne\emptyset\iff \bm b\in B_{\pi(S^c)},\quad \forall\, S\subseteq E(G).
\end{equation}

After introducing the necessary preliminary concepts, we now present our second main result. In this result, we use the restricted arrangements $\tilde{\mathcal{A}}_{E(G)}^{\bm b}$ and $\mathcal{A}_{E(G)}^{\bm b}$  to describe nowhere-zero affine flows.  We then show that the three polynomials $\chi(\tilde{\mathcal{A}}_{E(G)}^{\bm b},t)$, $\varphi_{G\setminus E_{\rm cut}^{\bm b}(G)}(\alpha,t)$ with $\alpha=\alpha_{G\setminus E_{\rm cut}^{\bm b}(G),\bm b}$ and $\varphi_G(\tilde{\Pi}_{\bm b},t)$ are identical, where
\[
E_{\rm cut}^{\bm b}(G):=\big\{e\in E_{\rm cut}(G):\bm b\in B(G\setminus e,\mathbb{F})\big\}.
\]
Likewise, the polynomials $\chi(\mathcal{A}_{E(G)}^{\bm b},t)$, $\varphi_G(\alpha,t)$ with $\alpha=\alpha_{G,\bm b}$ and $\varphi_G(\Pi_{\bm b},t)$ are also identical.
\begin{theorem}[Counting Pattern of Nowhere-Zero Affine Flows]\label{thm:nowhere-zero affine flow counting}
Let $\mathbb{F}$ be a field and $\bm b\in B(G,\mathbb{F})$.  Then,
\begin{itemize}
\item [{\rm (a)}]
$\chi(\tilde{\mathcal{A}}_{E(G)}^{\bm b},t)$ is a polynomial of degree $n(G)$ with leading coefficient $1$. 
\item[{\rm (b)}] $\chi(\tilde{\mathcal{A}}_{E(G)}^{\bm b},t)=\varphi_{G\setminus E_{\rm cut}^{\bm b}(G)}(\alpha,t)=\varphi_G(\tilde{\Pi}_{\bm b},t)$ with $\alpha=\alpha_{G\setminus E_{\rm cut}^{\bm b}(G),\bm b}$, where $\tilde{\Pi}_{\bm b}:=\big\{\pi\in\tilde{\Pi}(G):\bm b\in B_{\pi}\big\}$. In particular, $\varphi_G(\tilde{\Pi}_{\bm 0},t)$ is the flow polynomial $\varphi_{G\setminus E_{\rm cut}(G)}(t)$.
\item[{\rm (c)}]If $\mathbb{F}=\mathbb{F}_q$ is a finite field of $q$ elements, then
\[
\chi(\tilde{\mathcal{A}}_{E(G)}^{\bm b},q)=|M(\tilde{\mathcal{A}}_{E(G)}^{\bm b})|=|F_{\rm nz}(G\setminus E^{\bm b}_{\rm cut}(G),\bm b;\mathbb{F}_q)|=\varphi_G(\tilde{\Pi}_{\bm b},q).
\]
\end{itemize}
\end{theorem}

\begin{theorem}\label{thm:nowhere-zero affine flow counting1}
Let $\mathbb{F}$ be a field and $\bm b\in B(G,\mathbb{F})$. Then,
\begin{itemize}
\item[{\rm (a)}] $\chi(\mathcal{A}_{E(G)}^{\bm b},t)=\chi(\tilde{\mathcal{A}}_{E(G)}^{\bm b},t)$ is a polynomial of degree $n(G)$ with leading coefficient $1$ if $E_{\rm cut}^{\bm b}(G)=\emptyset$, and vanishes identically otherwise. 
\item[{\rm (b)}] $\chi(\mathcal{A}_{E(G)}^{\bm b},t)=\varphi_G(\alpha,t)=\varphi_G(\Pi_{\bm b},t)$ with $\alpha=\alpha_{G,\bm b}$, where $\Pi_{\bm b}:=\big\{\pi\in\Pi(G):\bm b\in B_\pi\big\}$. In particular, $\varphi_G(\Pi_{\bm 0},t)$ is the flow polynomial $\varphi_G(t)$.
\item[{\rm (c)}]If $\mathbb{F}=\mathbb{F}_q$ is a finite field of $q$ elements,  then
\[
\chi(\mathcal{A}_{E(G)}^{\bm b},q)=|M(\mathcal{A}_{E(G)}^{\bm b})|=|F_{\rm nz}(G,\bm b;\mathbb{F}_q)|=\varphi_G(\Pi_{\bm b},q).
\]
\end{itemize}
\end{theorem}
Parts (c) of  \autoref{thm:nowhere-zero affine flow counting} and \autoref{thm:nowhere-zero affine flow counting1} show that  the polynomial $\varphi_G(\Pi_{\bm b},t)$ counts the number of nowhere-zero affine flows with boundary $\bm b$, and $\varphi_G(\tilde{\Pi}_{\bm b},t)$ enumerates nowhere-zero affine flows after deleting the cut-edges that are forced to carry value zero. Accordingly, we refer to $\varphi_G(\tilde{\Pi}_{\bm b},t)$  and $\varphi_G(\Pi_{\bm b},t)$ as {\em affine flow polynomials}.

To further study the combinatorial properties of nowhere-zero affine flows, we now introduce two types of boundary arrangements in $B(G,\mathbb{F})$ that correspond to special vertex partitions of $G$ arising from its bond structure. These geometric objects serve as key ingredients in our classification of restricted arrangements and  partition-selected flow polynomials.  The {\em boundary arrangement} $\mathcal{A}_{\Pi(G)}$  is a hyperplane arrangement in $B(G,\mathbb{F})$ defined by 
\[
\mathcal{A}_{\Pi(G)}:=\big\{B_\pi:\pi\in\Pi(G), |\pi|=|\pi(E(G))|+1\big\}.
\]
Similarly, let $\tilde{\Pi}(G)$ denote the set of partitions $\pi\in\Pi(G)$ where the ends of each cut-edge are contained in a block of $\pi$, i.e.,
\[
\tilde{\Pi}(G):=\big\{\pi(S)\in \Pi(G):E_{\rm cut}(G)\subseteq S\big\}.
\]
The {\em reduced boundary arrangement} $\mathcal{A}_{\tilde{\Pi}(G)}$ is a hyperplane arrangement in $B(G,\mathbb{F})$ defined as  
\[
\mathcal{A}_{\tilde{\Pi}(G)}:=\big\{B_\pi:\pi\in\tilde{\Pi}(G), |\pi|=|\pi(E(G))|+1\big\}.
\]
It is worth noting that the reduced boundary arrangement $\mathcal{A}_{\tilde{\Pi}(G)}$ is the subarrangement  of $\mathcal{A}_{\Pi(G)}$ obtained by deleting all hyperplanes $B_{\pi(G\setminus e)}$ for $e\in E_{\rm cut}(G)$.

Our third result gives a classification of the restricted arrangements $\tilde{\mathcal{A}}_{E(G)}^{\bm b}$ and $\mathcal{A}_{E(G)}^{\bm b}$, based respectively on the reduced boundary arrangement and the boundary arrangement.
\begin{theorem}[Geometric Classification]\label{Classification}Let $\bm b_1,\bm b_2\in B(G,\mathbb{F})$.
\begin{itemize}
\item[{\rm(a)}] If $\bm b_1,\bm b_2\in M(\mathcal{A}_{\tilde{\Pi}(G)}/\tilde{X})$ for some flat $\tilde{X}\in L(\mathcal{A}_{\tilde{\Pi}(G)})$, then
\[
L(\tilde{\mathcal{A}}_{E(G)}^{\bm b_1})\cong L(\tilde{\mathcal{A}}_{E(G)}^{\bm b_2}),\quad\quad\tilde{\Pi}_{\bm b_1}=\tilde{\Pi}_{\bm b_2}
\]
and 
\[
\chi(\tilde{\mathcal{A}}_{E(G)}^{\bm b_1},t)=\chi(\tilde{\mathcal{A}}_{E(G)}^{\bm b_2},t)=\varphi_G(\tilde{\Pi}_{\bm b_1},t)=\varphi_G(\tilde{\Pi}_{\bm b_2},t).
\]
\item[{\rm(b)}]If $\bm b_1,\bm b_2\in M(\mathcal{A}_{\Pi(G)}/X)$ for some flat $X\in L(\mathcal{A}_{\Pi(G)})$, then
\[
L(\mathcal{A}_{E(G)}^{\bm b_1})\cong L(\mathcal{A}_{E(G)}^{\bm b_2}),\quad\quad\Pi_{\bm b_1}=\Pi_{\bm b_2}
\]
and 
\[
\chi(\mathcal{A}_{E(G)}^{\bm b_1},t)=\chi(\mathcal{A}_{E(G)}^{\bm b_2},t)=\varphi_G(\Pi_{\bm b_1},t)=\varphi_G(\Pi_{\bm b_2},t).
\]
\end{itemize}
\end{theorem}

It is well known that the nonzero coefficients of the characteristic polynomial $\chi(\mathcal{A},t)$ are nonzero and alternate in sign, unless $\chi(\mathcal{A},t)\equiv0$. According to \autoref{thm:nowhere-zero affine flow counting} and \autoref{thm:nowhere-zero affine flow counting1}, the polynomials $\varphi_G(\tilde{\Pi}_{\bm b},t)$ and  $\varphi_G(\Pi_{\bm b},t)$ can be written respectively as:
\[
\varphi_G(\tilde{\Pi}_{\bm b},t)=\sum_{i=0}^{n(G)}(-1)^iw_i(\tilde{\Pi}_{\bm b})t^{n(G)-i}\quad\And\quad \varphi_G(\Pi_{\bm b},t)=\sum_{i=0}^{n(G)}(-1)^iw_i(\Pi_{\bm b})t^{n(G)-i},
\]
with all coefficients $w_i(\tilde{\Pi}_{\bm b})$ and $w_i(\Pi_{\bm b})$ being nonnegative.

As our fourth main result, we establish unified comparison relations for the unsigned coefficients of affine flow polynomials, based on the intersection posets $L(\mathcal{A}_{\tilde{\Pi}(G)})$ and $L(\mathcal{A}_{\Pi(G)})$.
\begin{theorem}[Comparison of Coefficients]\label{thm:comparison}
Let $\bm b_1, \bm b_2\in B(G,\mathbb{F})$. 
\begin{itemize} 
\item[{\rm (a)}] If $\bm b_1\in M(\mathcal{A}_{\tilde{\Pi}(G)}/\tilde{X}_1)$ and  $\bm b_2\in M(\mathcal{A}_{\tilde{\Pi}(G)}/\tilde{X}_2)$ with $\tilde{X}_1\subseteq\tilde{X}_2$ in $L(\mathcal{A}_{\tilde{\Pi}(G)})$, then 
\[
w_i(\tilde{\Pi}_{\bm b_1})\le w_i(\tilde{\Pi}_{\bm b_2}) \quad\text{ for } i=0,1,\ldots,n(G).
\] 
\item[{\rm (b)}]If $\bm b_1\in M(\mathcal{A}_{\Pi(G)}/X_1)$ and  $\bm b_2\in M(\mathcal{A}_{\Pi(G)}/X_2)$ with $X_1\subseteq X_2$ in $L(\mathcal{A}_{\Pi(G)})$, then 
\[
w_i(\Pi_{\bm b_1})\le w_i(\Pi_{\bm b_2}) \quad\text{ for } i=0,1,\ldots,n(G).
\] 
\end{itemize}
\end{theorem}

\autoref{Classification} says that for any $\tilde{X}\in L(\mathcal{A}_{\tilde{\Pi}(G)})$, $\varphi_G(\tilde{\Pi}_{\bm b},t)$ are the same polynomial for all $\bm b\in M(\mathcal{A}_{\tilde{\Pi}(G)}/\tilde{X})$, denoted by $\tilde{\varphi}_G(\tilde{X},t)$; and for any $X\in L(\mathcal{A}_{\Pi(G)})$, $\varphi_G(\Pi_{\bm b},t)$ are also the same polynomial for all $\bm b\in M(\mathcal{A}_{\Pi(G)}/X)$, denoted by  $\varphi_G(X,t)$. With the above notations, our final main result provides two decomposition formulas for nowhere-zero affine flows associated with the boundary arrangement and the reduced boundary arrangement, respectively. 
\begin{theorem}[Decomposition Formulas]\label{Dec1} Let $\mathbb{F}$ be a field. 
\begin{itemize}
\item[{\rm (a)}]
If $\mathbb{F}$ is an infinite field, then
\[
(t-1)^{|E(G)|}=\sum_{X\in L(\mathcal{A}_{\Pi(G)})}\varphi_G(X,t)\,\chi(\mathcal{A}_{\Pi(G)}/X,t).
\]
If $\mathbb{F}$ is a finite field of $q$ elements, then
\[
(q-1)^{|E(G)|}=\sum_{X\in L(\mathcal{A}_{\Pi(G)})}\varphi_G(X,q)\,\chi(\mathcal{A}_{\Pi(G)}/X,q).
\]
\item[{\rm (b)}]
If $\mathbb{F}$ is an infinite field, then
\[
t^{|E_{\rm cut}(G)|}(t-1)^{|E_{\rm cyc}(G)|}=\sum_{\tilde{X}\in L(\mathcal{A}_{\tilde{\Pi}(G)})}\tilde{\varphi}_G(\tilde{X},t)\,\chi(\mathcal{A}_{\tilde{\Pi}(G)}/\tilde{X},t).
\]
If $\mathbb{F}$ is a finite field of $q$ elements, then
\[
q^{|E_{\rm cut}(G)|}(q-1)^{|E_{\rm cyc}(G)|}=\sum_{\tilde{X}\in L(\mathcal{A}_{\tilde{\Pi}(G)})}\tilde{\varphi}_G(\tilde{X},q)\,\chi(\mathcal{A}_{\tilde{\Pi}(G)}/\tilde{X},q).
\]
\end{itemize}
\end{theorem}

\section{Proof of \autoref{Partition-CTFT-DC}}\label{Sec-2}
This section focuses on investigating partition-selected graph polynomials. We first derive the specialization relations for partition-selected graph polynomials from their definitions.
\begin{theorem}\label{Partition-Relations-CTF-Tutte}
Let $\mathfrak{p}\subseteq\Pi(G)$. The partition-selected chromatic, tension, flow polynomials are respectively related to the partition-selected Tutte polynomial by the following relations:
\begin{align}
&\chi_G(\mathfrak{p},t)=(-1)^{r(G)}t^{c(G)}T_G(\mathfrak{p};1-t,0),\label{Partition-Chromatic-Tutte-Relation}\\
&\tau_G(\mathfrak{p},t)=(-1)^{r(G)}T_G(\mathfrak{p};1-t,0),\label{Partition-Tension-Tutte-Relation}\\
&\varphi_G(\mathfrak{p},t)=(-1)^{n(G)}T_G(\mathfrak{p};0,1-t)\label{Partition-Flow-Tutte-Relation}.
\end{align}
\end{theorem}
Let $\mathcal{E}(G,\mathfrak{p})$ denote the class of edge sets $S\subseteq E(G)$ with $\pi(S)\in\mathfrak{p}$; equivalently, each sum in \eqref{Partition-Chromatic-Def}, \eqref{Partition-Tension-Def}, \eqref{Partition-Flow-Def}, and \eqref{Partition-Rank-Def} is taken over this class. Associated with an edge $e$ of $G$, the class $\mathcal{E}(G,\mathfrak{p})$ is partitioned into two subclasses according to whether $e$ is present or absent:
\begin{equation}\label{Edge-Partition}
\mathcal{E}_e(G,\mathfrak{p})=\big\{S\in\mathcal{E}(G,\mathfrak{p}):e\in S\big\}\quad\And\quad \mathcal{E}_{\bar{e}}(G,\mathfrak{p})=\big\{S\in\mathcal{E}(G,\mathfrak{p}):e\notin S\big\}.
\end{equation}

We present a deletion-contraction formula for partition-selected rank generating function.
\begin{theorem}\label{Partition-Rank-DC}
Let $\mathfrak{p}\subseteq\Pi(G)$ and $e\in E(G)$. The partition-selected rank generating function $R_G(\mathfrak{p},x,y)$ satisfies
\[
R_G(\mathfrak{p};x,y)=\begin{cases}
(y+1)R_{G'}(\mathfrak{p}';x,y),&\text{ if } e \text{ is a loop};\\
xR_{G'}(\mathfrak{p}';x,y)+R_{G^{''}}(\mathfrak{p}^{''};x,y),&\text{ if } e \text{ is a cut-edge};\\
R_{G'}(\mathfrak{p}';x,y)+R_{G^{''}}(\mathfrak{p}^{''};x,y),&\text{ otherwise}.
\end{cases}
\]
\end{theorem}
\begin{proof}
According to \eqref{Edge-Partition},  $R_G(\mathfrak{p};x,y)$ can be written as the sum of two parts:
\begin{equation}\label{Partition-Rank-Two}
R_G(\mathfrak{p};x,y)=\sum_{S\in \mathcal{E}_e(G,\mathfrak{p})}x^{r(G)-r(S)}y^{n(S)}+\sum_{S\in \mathcal{E}_{\bar{e}}(G,\mathfrak{p})}x^{r(G)-r(S)}y^{n(S)}.
\end{equation}
Clearly, we have the simple relation
\begin{equation*}\label{Partition-e}
\mathcal{E}_{\bar{e}}(G,\mathfrak{p})=\mathcal{E}(G',\mathfrak{p}').
\end{equation*}
In addition, the class $\mathcal{E}_e(G,\mathfrak{p})$ can be identified with the class $\mathcal{E}(G^{''},\mathfrak{p}^{''})$ by identifying $S$ members of the former with $S/e$ members of the latter, where $S/e$ denotes the edge set of the graph $(G|S)/e$. In this context, each member $S$ in $\mathcal{E}(G^{''},\mathfrak{p}^{''})$ corresponds bijectively to a unique member $S\sqcup e$ in $\mathcal{E}_e(G,\mathfrak{p})$.

Write $r'$ and $n'$ for the rank and nullity functions in $G'$, and $r^{''}$ and $n^{''}$ for the rank and nullity functions in $G^{''}$. If $e$ is a loop, then $G'=G^{''}$ and $\mathfrak{p}'=\mathfrak{p}^{''}=\mathfrak{p}$. For any $S\subseteq E(G)\setminus e$, 
\begin{equation*}\label{Rank-Nullity-Relation1}
r(G)=r'(G'),\quad r(S\sqcup e)=r(S)=r'(S)\quad\And\quad n(S\sqcup e)=n(S)+1=n'(S)+1.
\end{equation*}
It follows that the sum over $\mathcal{E}_e(G,\mathfrak{p})$ equals $yR_{G'}(\mathfrak{p}';x,y)$, and the sum over $\mathcal{E}_{\bar{e}}(G,\mathfrak{p})$ equals $R_{G'}(\mathfrak{p}';x,y)$ as stated in \eqref{Partition-Rank-Two}. Thus
\[
R_G(\mathfrak{p};x,y)=(y+1)R_{G'}(\mathfrak{p}';x,y).
\]

If $e$ is a cut-edge, then  for any $S\subseteq E(G)\setminus e$, we have
\[
r(G)=r'(G')+1,\quad r(S)=r'(S),\quad n(S)=n'(S)
\]
and
\[
r(G)=r^{''}(G^{''})+1,\quad r(S\sqcup e)=r^{''}(S)+1,\quad n(S\sqcup e)=n^{''}(S).
\]
Then we deduce that the sum over $\mathcal{E}_e(G,\mathfrak{p})$ is $R_{G^{''}}(\mathfrak{p}^{''};x,y)$, and the sum over $\mathcal{E}_{\bar{e}}(G,\mathfrak{p})$ is $xR_{G'}(\mathfrak{p}';x,y)$ in \eqref{Partition-Rank-Two}. Hence, we obtain
\[
R_G(\mathfrak{p};x,y)=xR_{G'}(\mathfrak{p}';x,y)+R_{G^{''}}(\mathfrak{p}^{''};x,y).
\]

If $e$ is neither a loop nor a cut-edge, for any $S\subseteq E(G)\setminus e$, we have
\[
r(G)=r'(G'),\quad r(S)=r'(S),\quad n(S)=n'(S)
\]
and
\[
r(G)=r^{''}(G^{''})+1,\quad r(S\sqcup e)=r^{''}(S)+1,\quad n(S\sqcup e)=n^{''}(S).
\]
It follows that the sum over $\mathcal{E}_e(G,\mathfrak{p})$ is $R_{G^{''}}(\mathfrak{p}^{''};x,y)$, and the sum over $\mathcal{E}_{\bar{e}}(G,\mathfrak{p})$ is $R_{G'}(\mathfrak{p}';x,y)$ in \eqref{Partition-Rank-Two}. Therefore, we arrive at
\[
R_G(\mathfrak{p};x,y)=R_{G'}(\mathfrak{p}';x,y)+R_{G^{''}}(\mathfrak{p}^{''};x,y)
\]
in this case, which completes the proof.
\end{proof}

Using \autoref{Partition-Rank-DC}, we proceed to prove \autoref{Partition-CTFT-DC}.
\begin{proof}[Proof of \autoref{Partition-CTFT-DC}]
Together with  $T_G(\mathfrak{p};x,y)=R_G(\mathfrak{p};x-1,y-1)$ in \eqref{Partition-Tutte-Def} and \autoref{Partition-Rank-DC}, we obtain \eqref{Partition-Tutte-DC}. Applying \eqref{Partition-Tutte-DC} to \eqref{Partition-Chromatic-Tutte-Relation} (\eqref{Partition-Tension-Tutte-Relation}, \eqref{Partition-Flow-Tutte-Relation}, resp.), we can directly deduce
\eqref{Partition-Chromatic-DC} (\eqref{Partition-Tension-DC}, \eqref{Partition-Flow-DC}, resp.). Finally, since $r(G)-r(S)=c(S)-c(G)$, from \eqref{Partition-Chromatic-Def} and  \eqref{Partition-Tension-Def}, we have the relation $\chi_G(\mathfrak{p},t)=t^{c(G)}\tau_G(\mathfrak{p},t)$, which completes the proof.
\end{proof}

\section{Proofs of \autoref{thm:nowhere-zero affine flow counting} and \autoref{thm:nowhere-zero affine flow counting1}}\label{Sec-3}
\subsection{Proof of Theorems \ref{thm:nowhere-zero affine flow counting} and \ref{thm:nowhere-zero affine flow counting1} }\label{Sec3-1}
We begin by characterizing when the affine subspaces $H_e^{\bm b}$ are hyperplanes in $F(G,\bm b;\mathbb{F})$ via cut-edges and cycle-edges.
\begin{lemma}\label{prop:characterization of cosets}
Let $\bm b\in B(G,\mathbb{F})$. Then, for a cut-edge $e\in E_{\rm cut}(G)$, $H_e^{\bm b}=F(G,\bm b;\mathbb{F})$ if $e\in E_{\rm cut}^{\bm b}(G)$, and $H_e^{\bm b}=\emptyset$ otherwise; for a cycle-edge $e\in E_{\rm cyc}(G)$, $H_e^{\bm b}$ is a hyperplane in $F(G,\bm b;\mathbb{F})$.
\end{lemma}
\begin{proof}
Fix $e\in E(G)$. If $H_e^{\bm b}\ne\emptyset$, then for any fixed $\bm c_0\in H_e^{\bm b}$,  we have
\begin{equation}\label{eq:affine representation}
H_e^{\bm b}=\bm c_0+H_e^{\bm 0}\quad\text{and}\quad F(G,\bm b;\mathbb{F})=\bm c_0+F(G,\bm 0;\mathbb{F}).
\end{equation}
If $e$ is a cut-edge, then every flow vanishes on $e$. So, $H_e^{\bm 0}=F(G,\bm 0;\mathbb{F})$. According to \eqref{eq:affine representation}, $H_e^{\bm b}=F(G,\bm b;\mathbb{F})$ whenever $H_e^{\bm b}\ne\emptyset$. Notably, $H_e^{\bm b}\ne\emptyset$ if and only if there exists $\bm c\in \mathbb{F}^{E(G)}$ with $\partial\bm c=\bm b$ and $\bm c(e)=0$. This is equivalent to $\bm b\in B(G\setminus e,\mathbb{F})$. If $e$ is a cycle-edge, then there exists a flow $\bm c_1\in F(G,\bm 0;\mathbb{F})$ with $\bm c_1(e)\ne 0$. This means that $H_e^{\bm 0}$ is a hyperplane in $F(G,\bm 0;\mathbb{F})$. Since $\bm b\in B(G,\mathbb{F})$, there is $\bm c\in F(G,\bm b;\mathbb{F})$. Consequently,  $\bm c-\frac{\bm c(e)}{\bm c_1(e)}\bm c_1\in H_e^{\bm b}$, hence $H_e^{\bm b}\ne\emptyset$. By \eqref{eq:affine representation}, $H_e^{\bm b}$ is a hyperplane in $F(G,\bm b;\mathbb{F})$.
\end{proof}
\autoref{prop:characterization of cosets} shows that only cycle-edges give rise to hyperplanes in  $\tilde{\mathcal{A}}_{E(G)}^{\bm b}$. Thus, we have
\[
\tilde{\mathcal{A}}_{E(G)}^{\bm b}=\big\{H_e^{\bm b}: e\in E_{\rm cyc}(G)\big\}.
\]
Hence, the intersection poset $L\big(\tilde{\mathcal{A}}_{E(G)}^{\bm b}\big)$ consists precisely of the nonempty intersections of hyperplanes $H_e^{\bm b}$ induced by cycle-edges $e$. It follows from \eqref{eq:equivalence for coset intersection nonempty} that  the intersection poset $L\big(\tilde{\mathcal{A}}_{E(G)}^{\bm b}\big)$ can be explicitly expressed in the form
\[
L\big(\tilde{\mathcal{A}}_{E(G)}^{\bm b}\big)=\big\{H_S^{\bm b}: S\subseteq E_{\rm cyc}(G)\And \bm b\in B_{\pi(S^c)}\big\}.
\] 

Notice that $H_S^{\bm b}$ is the solution space of the linear system $M_G\bm c=\bm b$ with the constraints $\bm c(e)=0$ for all $e\in S$. When $H_S^{\bm b}\ne\emptyset$, we have
\begin{equation}\label{eq:dimension of stratum}
\dim H_S^{\bm b}=|E(G)|-{\rm rank}\Big(\begin{bmatrix}M_G\\I_S\end{bmatrix}\Big)=n(S^c),
\end{equation}
where $I_S$ is the submatrix of the  identity matrix $I_{E(G)}$ consisting of the rows indexed by $S$.  By \eqref{eq:equivalence for coset intersection nonempty} and \eqref{eq:dimension of stratum}, the characteristic polynomial $\chi(\tilde{\mathcal{A}}_{E(G)}^{\bm b},t)$ of $\tilde{\mathcal{A}}_{E(G)}^{\bm b}$ can be expressed as
\begin{equation}\label{eq:characteristic polynomial of affine flow arrangement}
\chi(\tilde{\mathcal{A}}_{E(G)}^{\bm b},t)=\sum_{S\subseteq E_{\rm cyc}(G),\,\bm b\in B_{\pi(S^c)}}(-1)^{|S|}t^{n(S^c)}.
\end{equation}

\begin{proof}[Proof of \autoref{thm:nowhere-zero affine flow counting}]
Proof of part (a). Note that for any $S\subseteq E_{\rm cyc}(G)$, $n(S^c)=n(G)$ if $S=\emptyset$, and $n(S^c)<n(G)$ otherwise. Since $\bm b\in B(G,\mathbb{F})=B_{\pi(\emptyset^c)}$, it follows from \eqref{eq:dimension of stratum} that $\chi(\tilde{\mathcal{A}}_{E(G)}^{\bm b},t)$ has degree $n(G)$ with leading coefficient $1$.

Proof of part (b). Recall that the partition-selected flow polynomial $\varphi_G(\tilde{\Pi}_{\bm b},t)$ is  \[\varphi_G(\tilde{\Pi}_{\bm b},t)=\sum_{S^c\subseteq E(G),\,\pi(S^c)\in \tilde{\Pi}_{\bm b}}(-1)^{|S|}t^{n(S^c)}.\] Since $\tilde{\Pi}_{\bm b}\subseteq\tilde{\Pi}(G)=\big\{\pi(S)\in \Pi(G):E_{\rm cut}(G)\subseteq S\big\}$, it follows that 
\[
\tilde{\Pi}_{\bm b}=\big\{\pi(S)\in \Pi(G):E_{\rm cut}(G)\subseteq S,\bm b\in B_{\pi(S)}\big\}.
\]
According to this, we can rewrite $\varphi_G(\tilde{\Pi}_{\bm b},t)$ as 
\[
\varphi_G(\tilde{\Pi}_{\bm b},t)=\sum_{S\subseteq E_{\rm cyc}(G),\,\bm b\in B_{\pi(S^c)}}(-1)^{|S|}t^{n(S^c)}.
\]
By \eqref{eq:characteristic polynomial of affine flow arrangement}, we have $\chi(\tilde{\mathcal{A}}_{E(G)}^{\bm b},t)=\varphi_G(\tilde{\Pi}_{\bm b},t)$. When $\bm b=\bm 0$, we have $\tilde{\Pi}_{\bm 0}=\tilde{\Pi}(G)=\Pi(G\setminus E_{\rm cut}(G))$. Thus, $\varphi_G(\tilde{\Pi}_{\bm 0},t)=\varphi_{G\setminus E_{\rm cut}(G)}(t)$.

On the other hand, the assigning polynomial in \eqref{AP} is given by
\[
\varphi_G(\alpha,t)=\sum_{S\subseteq E(G),\,G\setminus S\text{ is $\bm b$-compatible}}(-1)^{|S|}t^{n(S^c)}.
\]
Note from \cite[Proposition 2.1]{Jaeger1992} that for any $\bm b\in \mathbb{F}^{V(G)}$, 
\begin{equation}\label{CBS}
\bm b\in B(G,\mathbb{F})\iff \sum_{v\in V(H)}\bm b(v)=0 \text{ for each component $H$ of $G$},
\end{equation} 
see \autoref{prop:characterization of boundary space}. Combining \eqref{Identity}, the condition that $G\setminus S$ is $\bm b$-compatible is equivalent to $\bm b\in B_{\pi(S^c)}$. Thus, $\varphi_{G\setminus E_{\rm cut}^{\bm b}(G)}(\alpha,t)$ can be written as
\[
\varphi_{G\setminus E_{\rm cut}^{\bm b}(G)}(\alpha,t)=\sum_{S\subseteq E(G\setminus E_{\rm cut}^{\bm b}(G)),\,\bm b\in B_{\pi(S^c)}}(-1)^{|S|}t^{n(S^c)},
\]
where $S^c=E(G)\smallsetminus E_{\rm cut}^{\bm b}(G)\smallsetminus S$.
Notice that for any $e\in E_{\rm cut}(G)\smallsetminus E_{\rm cut}^{\bm b}(G)$, $\bm b\notin B_{\pi(E(G)\smallsetminus e)}$ by  \autoref{prop:characterization of cosets}. Thus, the condition that $S\subseteq E(G\setminus E_{\rm cut}^{\bm b}(G))$ and $\bm b\in B_{\pi(S^c)}$ implies 
$S\subseteq E_{\rm cyc}(G)$. Applying \autoref{prop:characterization of cosets} again, we deduce $H_{E_{\rm cut}^{\bm b}(G)}^{\bm b}=F(G,\bm b;\mathbb{F})$. It follows that for any $S\subseteq E_{\rm cyc}(G)$, $H_{E_{\rm cut}^{\bm b}(G)\sqcup S}^{\bm b}\ne\emptyset$ if and only if $H_S^{\bm b}\ne\emptyset$.  Combining \eqref{eq:equivalence for coset intersection nonempty}, we have 
\[
\big\{S: S\subseteq E_{\rm cyc}(G), \bm b\in B_{\pi(E(G)\smallsetminus E_{\rm cut}^{\bm b}\smallsetminus S)}\big\}=\big\{S: S\subseteq E_{\rm cyc}(G), \bm b\in B_{\pi(E(G)\smallsetminus S)}\big\}.
\]
Following this,  we can rewrite $\varphi_{G\setminus E_{\rm cut}^{\bm b}(G)}(\alpha,t)$ in the form
\[
\varphi_{G\setminus E_{\rm cut}^{\bm b}(G)}(\alpha,t)=\sum_{S\subseteq E_{\rm cyc}(G),\,\bm b\in B_{\pi(E(G)\smallsetminus S)}}(-1)^{|S|}t^{n(E(G)\smallsetminus E_{\rm cut}^{\bm b}(G)\smallsetminus S)}.
\]
As $n(E(G)\smallsetminus S)=n(E(G)\smallsetminus E_{\rm cut}^{\bm b}(G)\smallsetminus S)$, we conclude $\chi(\tilde{\mathcal{A}}_{E(G)}^{\bm b},t)=\varphi_{G\setminus E_{\rm cut}^{\bm b}(G)}(\alpha,t)$ by \eqref{eq:characteristic polynomial of affine flow arrangement}. 

Proof of part (c). Applying part (b) to \eqref{AP}, we directly deduce
\[
\chi(\tilde{\mathcal{A}}_{E(G)}^{\bm b},q)=\varphi_G(\tilde{\Pi}_{\bm b},q)=\varphi_{G\setminus E_{\rm cut}^{\bm b}(G)}(\alpha,q)=|F_{\rm nz}(G\setminus E_{\rm cut}^{\bm b}(G),\bm b;\mathbb{F}_q)|.
\]
It remains to prove $|M(\tilde{\mathcal{A}}_{E(G)}^{\bm b})|=|F_{\rm nz}(G\setminus E^{\bm b}_{\rm cut}(G),\bm b;\mathbb{F}_q)|$. Note from \autoref{prop:characterization of cosets} that $H_e\cap F(G,\bm b;\mathbb{F}_q)=F(G,\bm b;\mathbb{F}_q)$ for any $e\in E_{\rm cut}^{\bm b}(G)$, and $H_e\cap F(G,\bm b;\mathbb{F}_q)=\emptyset$ for any $e\in E_{\rm cut}(G)\smallsetminus E_{\rm cut}^{\bm b}(G)$ . This implies that each element of $F(G,\bm b;\mathbb{F}_q)$ vanishes on $E^{\bm b}_{\rm cut}$, and is nowhere-zero on $E_{\rm cut}(G)\smallsetminus E_{\rm cut}^{\bm b}(G)$. Therefore, all elements of  $M(\tilde{\mathcal{A}}_{E(G)}^{\bm b})$ vanish on $E^{\bm b}_{\rm cut}$, are nowhere-zero on the remaining edges, and have boundary $\bm b$. Consequently, we have $M(\tilde{\mathcal{A}}_{E(G)}^{\bm b})\cong F_{\rm nz}(G\setminus E^{\bm b}_{\rm cut},\bm b;\mathbb{F}_q)$. This completes the proof.
\end{proof}

According to \autoref{prop:characterization of cosets}, the two restricted arrangements $\mathcal{A}_{E(G)}^{\bm b}$ and $\tilde{\mathcal{A}}_{E(G)}^{\bm b}$ satisfy  the following close relationship:
\[
\mathcal{A}_{E(G)}^{\bm b}=
\begin{cases}\tilde{\mathcal{A}}_{E(G)}^{\bm b},&\text{if }E_{\rm cut}^{\bm b}(G)=\emptyset;\\
\tilde{\mathcal{A}}_{E(G)}^{\bm b}\sqcup\{F(G,\bm b;\mathbb{F})\},&\text{if }E_{\rm cut}^{\bm b}(G)\ne\emptyset.
\end{cases}
\]
Consequently, we have
\[
\chi(\mathcal{A}_{E(G)}^{\bm b},t)=
\begin{cases}
\chi(\tilde{\mathcal{A}}_{E(G)}^{\bm b},t),&\text{if }E_{\rm cut}^{\bm b}(G)=\emptyset;\\
0,&\text{if }E_{\rm cut}^{\bm b}(G)\ne\emptyset.
\end{cases}
\]
Thus, part (a) of \autoref{thm:nowhere-zero affine flow counting1} follows directly from part (a)  of  \autoref{thm:nowhere-zero affine flow counting}. The proofs of the remaining parts of \autoref{thm:nowhere-zero affine flow counting1} are similar to those of parts (b) and (c) of \autoref{thm:nowhere-zero affine flow counting}, and require only minor modifications. Therefore, we omit the detailed proofs  of \autoref{thm:nowhere-zero affine flow counting1}. 

\subsection{Deletion-contraction formulas}\label{Sec3-2}
Fix an edge $e$ of $G$ with ends $u$ and $v$. Recall that $G' = G\setminus e$ and $G'' = G/e$. We reduce each vertex chain $\bm b\in\mathbb{F}^{V(G)}$ of $G$ to a vertex chain $\bm b'\in\mathbb{F}^{V(G')}$ of $G'$ and a vertex chain $\bm b''\in\mathbb{F}^{V(G'')}$ of $G''$ in the following way: if $e$ is a loop, we set $\bm b'=\bm b''=\bm b$; if $e$ is an edge with distinct ends $u$ and $v$, we set
\[
\bm b':=\bm b,\quad\quad
\bm b''(w):=
\begin{cases}
\bm b(w), &\text{if } w\ne u,v;\\
\bm b(u)+\bm b(v), & \text{if } w=uv/e.
\end{cases}
\]
Associated with $\bm b'$ and $\bm b''$, and similar to $\Pi_{\bm b}$, we further consider the following two selections:
\[
\Pi_{\bm b'}:=\big\{\pi\in\Pi(G'):\bm b'\in B_\pi\big\},\quad \Pi_{\bm b''}:=\big\{\pi\in\Pi(G''):\bm b''\in B_\pi\big\}.
\]
It is natural to ask whether the affine flow polynomial $\varphi_G(\Pi_{\bm b},t)$ satisfies a deletion-contraction recurrence associated to affine flow polynomials $\varphi_G(\Pi_{\bm b'},t)$ and $\varphi_G(\Pi_{\bm b''},t)$. 

Recall from \eqref{Triple-Partitions} that 
\[
\Pi_{\bm b}'=\Pi_{\bm b}\cap\Pi(G')\quad\And\quad \Pi_{\bm b}''=\big\{\pi/e:\pi\in\Pi_{\bm b}\cap\Pi_e(G)\big\},
\]
where $\Pi_e(G)=\big\{\pi\in\Pi(G):\{u,v\} \text{ is contained in a block of } \pi\big\}$, as presented in \eqref{Pi-e}. Before proceeding further, we require the following lemma. 
\begin{lemma}\label{lem:deletion-contraction for selection}
With the above notations, we have
\[
\Pi_{\bm b}'=\Pi_{\bm b'}\quad\quad\text{and}\quad\quad\Pi_{\bm b}''=\Pi_{\bm b''}.
\]
\end{lemma}
\begin{proof}
Since $\Pi(G')\subseteq\Pi(G)$, $V(G')=V(G)$, and $\bm b'=\bm b$, it follows that 
\[
\Pi_{\bm b}'=\Pi_{\bm b}\cap\Pi(G')=\big\{\pi(S):S\subseteq E(G)\smallsetminus e, \bm b\in B_{\pi(S)}\big\}=\Pi_{\bm b'}.
\]
For the latter equation, when $e$ is a loop, it is straightforward to see that $\bm b''=\bm b$, $\Pi(G)=\Pi_e(G)$, and $B_{\pi(S)}=B_{\pi(S\smallsetminus e)}$.
Therefore, we deduce
\[
\Pi_{\bm b}''=\big\{\pi(S)/e:S\subseteq E(G), \bm b\in B_{\pi(S)}\big\}=\Pi_{\bm b''}.
\]
When $e$ is an edge with distinct ends $u$ and $v$, given any $\pi(S)\in\Pi_e(G)$, let $V_{uv}$ denote the block containing both $u$ and $v$. Then
$\sum_{w\in V_{uv}\setminus\{u,v\}}\bm b''(w)+\bm b''(uv/e)=\sum_{w\in V_{uv}}\bm b(w)$, and for any other block $W\ne V_{uv}$, we have $\sum_{w\in W}\bm b''(w)=\sum_{w\in W}\bm b(w)$. It follows from \eqref{CBS} that 
\[
\bm b\in B_{\pi(S)}\Longleftrightarrow\bm b''\in B_{\pi(S/e)}.
\]
Hence, $\pi(S)\in\Pi_{\bm b}\cap\Pi_e(G)$ implies $\pi(S)/e\in\Pi_{\bm b''}$, and thus $\Pi_{\bm b}''\subseteq\Pi_{\bm b''}$. Conversely, if $\pi(S)/e\in\Pi_{\bm b''}$ with $e\in S$, then $\bm b''\in B_{\pi(S/e)}$, and hence $\bm b\in B_{\pi(S)}$. Consequently, $\pi(S)/e\in\Pi_{\bm b}''$. This completes the proof of the second equation.
\end{proof}

We are now ready to state our desired deletion-contraction formulas for affine flow polynomials $\varphi_G(\Pi_{\bm b},t)$ and characteristic polynomials $\chi(\mathcal{A}_{E(G)}^{\bm b},t)$.
\begin{corollary}[Deletion-Contraction Formulas]\label{Deletion-Contraction}
The polynomials $\varphi_G(\Pi_{\bm b},t)$ and $\chi(\mathcal{A}_{E(G)}^{\bm b},t)$ satisfy the deletion-contraction recurrences:
\[
\varphi_G(\Pi_{\bm b},t)=\begin{cases}
(t-1)\varphi_{G'}(\Pi_{\bm b'},t),&\text{if $e$ is a loop};\\
\varphi_{G''}(\Pi_{\bm b''},t)-\varphi_{G'}(\Pi_{\bm b'},t),&\text{otherwise}.
\end{cases}
\]
and
\[
\chi(\mathcal{A}_{E(G)}^{\bm b},t)=\begin{cases}
(t-1)\chi(\mathcal{A}_{E(G')}^{\bm b'},t),&\text{if $e$ is a loop};\\
\chi(\mathcal{A}_{E(G'')}^{\bm b''},t)-\chi(\mathcal{A}_{E(G')}^{\bm b'},t),&\text{otherwise}.
\end{cases}
\]
\end{corollary}
\begin{proof}
According to \autoref{lem:deletion-contraction for selection}, we have
\[
\varphi_{G'}(\Pi_{\bm b'},t)=\varphi_{G'}(\Pi_{\bm b}',t)\quad\And\quad \varphi_{G''}(\Pi_{\bm b''},t)=\varphi_{G''}(\Pi_{\bm b}'',t).
\]
Combining \eqref{Partition-Flow-DC}, the first deletion-contraction recurrence holds for $\varphi_G(\Pi_{\bm b},t)$. Applying this to 
 \autoref{thm:nowhere-zero affine flow counting1}, the second deletion-contraction recurrence holds for $\chi(\mathcal{A}_{E(G)}^{\bm b},t)$.
\end{proof}

In general, the corresponding deletion-contraction recurrence in \autoref{Deletion-Contraction} does not hold for $\chi(\tilde{\mathcal{A}}_{E(G)}^{\bm b},t)$ and $\varphi_G(\tilde{\Pi}_{\bm b},t)$. For example, let $G$ be the $3$-cycle $C_3$, and $e$ be one of its edges. When $\bm b=\bm 0$, we can easily obtain 
\[
\chi(\tilde{\mathcal{A}}_{E(G)}^{\bm 0},t)=t-1,\quad \chi(\tilde{\mathcal{A}}_{E(G')}^{\bm 0'},t)=1,\quad \chi(\tilde{\mathcal{A}}_{E(G'')}^{\bm 0''},t)=t-1.
\]
We therefore see that $\chi(\tilde{\mathcal{A}}_{E(G)}^{\bm 0},t)\ne \chi(\tilde{\mathcal{A}}_{E(G'')}^{\bm 0''},t)-\chi(\tilde{\mathcal{A}}_{E(G')}^{\bm 0'},t)$. It follows from part (b) of \autoref{thm:nowhere-zero affine flow counting} that $\varphi_G(\tilde{\Pi}_{\bm b},t)\ne\varphi_{G''}(\tilde{\Pi}_{\bm b''},t)-\varphi_{G'}(\tilde{\Pi}_{\bm b'},t)$.
\section{Proofs of \autoref{Classification} and \autoref{thm:comparison}}\label{Sec4}
\subsection{Proof of \autoref{Classification}}\label{Sec4-1}
In this subsection, we are interested in classifying polynomials that count nowhere-zero affine flows and two types of restricted arrangements associated with $G$, as presented in \autoref{Classification}. Recall that $\Pi(G)$ is the collection of partitions of $V(G)$ corresponding to spanning subgraphs of $G$, i.e.,
\[\Pi(G)=\big\{\pi(S):S\subseteq E(G)\big\},\] 
and $\tilde{\Pi}(G)$ is the set of partitions $\pi\in\Pi(G)$ such that the ends of each cut-edge are contained in a block of $\pi$, i.e.,
\[
\tilde{\Pi}(G)=\big\{\pi(S)\in \Pi(G):E_{\rm cut}(G)\subseteq S\big\}.
\]
The corresponding boundary arrangement $\mathcal{A}_{\Pi(G)}$ and reduced boundary arrangement $\mathcal{A}_{\tilde{\Pi}(G)}$ are given respectively by
\[
\mathcal{A}_{\Pi(G)}=\big\{B_\pi:\pi\in\Pi(G), |\pi|=|\pi(E(G))|+1\big\}
\]
and
\[
\mathcal{A}_{\tilde{\Pi}(G)}=\big\{B_\pi:\pi\in\tilde{\Pi}(G), |\pi|=|\pi(E(G))|+1\big\}.
\]
The arrangements $\mathcal{A}_{\Pi(G)}$ and $\mathcal{A}_{\tilde{\Pi}(G)}$ are indeed hyperplane arrangements arising from the bond structure of $G$. An edge subset $F\subseteq E(G)$  is an {\em edge cut} in $G$ if there exists a partition \{X,Y\} of $V(G)$ such that $F=E[X,Y]$, where $E[X,Y]$ is the  set of edges of $G$ with one end in $X$ and the other end in $Y$. A minimal nonempty edge cut in $G$ is called a {\em bond}. To see this, we need a key characterization of boundary spaces originally due to Jaeger \cite{Jaeger1992}.  
\begin{proposition}[\cite{Jaeger1992}, Proposition 2.1]\label{prop:characterization of boundary space} 
The boundary space $B(G,\mathbb{F})$ consists of functions $\bm b\in \mathbb{F}^{V(G)}$ such that for each component $H$ of $G$,
\[
\sum_{v\in V(H)}\bm b(v)=0.
\]
\end{proposition}

\autoref{prop:characterization of boundary space} directly shows \eqref{Identity} in \autoref{Sec2-0}. If $\pi_1$ is finer than $\pi_2$ in $\Pi(G)$, then $B_{\pi_1}\subseteq B_{\pi_2}$ trivially holds. Since every partition $\pi\in\Pi(G)$ is finer than $\pi(E(G))$, it follows that every $B_\pi$ is contained in the boundary space $B(G,\mathbb{F})=B_{\pi(E(G))}$. Therefore, $B_\pi$ is indeed a subspace of $B(G,\mathbb{F})$. Moreover, note that $|\pi(S^c)|$ equals the number of connected components of $G\setminus S$. Thus, the condition $|\pi|=|\pi(E(G))|+1$ simply means that $\pi$ is obtained from $G$ by deleting a bond, i.e., $\pi$ is obtained from $\pi(E(G))$ by splitting a single block $V_j$ into two nonempty parts $V_{j_1}$ and $V_{j_2}$, with all other blocks remaining unchanged. Consequently, the defining equations for $B_\pi$ differ from those of $B_{\pi(E(G))}$ only in that the equation $\sum_{v\in V_j}\bm b(v)=0$ is replaced by the pair of equations $\sum_{v\in V_{j_1}}\bm b(v)=0$ and $\sum_{v\in V_{j_2}}\bm b(v)=0$. Thus, the dimension of each boundary subspace $B_\pi$ in $\mathcal{A}_{\Pi(G)}$ is exactly one dimension lower than the dimension of $B(G,\mathbb{F})$. Therefore, $\mathcal{A}_{\Pi(G)}$ is indeed a hyperplane arrangement in $B(G,\mathbb{F})$.

Fix an arbitrary $\pi\in\Pi(G)$. For each block $U$ of $\pi$, let $V$ be the vertex set of the unique connected component of $G$ containing $U$. If $U\ne V$, let $U_1,\ldots,U_{k_U}$ be the vertex sets of connected components of $G[V\smallsetminus U]$. Then, for each $j=1,\ldots,k_U$, $E[U,U_j]=E[V\smallsetminus U_j,U_j]$ is a bond of $G$, since both $G[U_j]$ and $G[V\smallsetminus U_j]$ are connected. Set $\pi_{U,j}:=\pi\big(E(G)\smallsetminus E[U,U_j]\big)$. Then $|\pi_{U,j}|=|\pi(E(G))|+1$, and hence $B_{\pi_{U,j}}\in\mathcal{A}_{\Pi(G)}$. Let $V_1,\ldots, V_{c(G)}$ be the vertex sets of connected components of $G$. We claim that 
\[
B_\pi=\bigcap_{U\in\pi,\, U\ne V_i,\,i=1,\ldots, c(G)}\bigcap_{j=1}^{k_U}B_{\pi_{U,j}}.
\] 
Since $\pi$ is finer than $\pi_{U,j}$ for all $U$ and $j$, we have $B_\pi\subseteq B_{\pi_{U,j}}$, and hence $B_\pi\subseteq\bigcap_{U\in\pi}\bigcap_{j=1}^{k_U}B_{\pi_{U,j}}$. Conversely, as each $U_j$ is a block of the corresponding partition $\pi_{U,j}$, we derive $\sum_{v\in U_j}\bm b(v)=0$. Together with $\sum_{v\in V}\bm b(v)=0$ and $V=U\sqcup U_1\sqcup\cdots\sqcup U_{k_U}$, we further deduce $\sum_{v\in U}\bm b(v)=0$. Hence the opposite inclusion also holds. Consequently, every boundary subspace $B_\pi$ for $\pi\in\Pi(G)$ is indeed the intersection of some hyperplanes from $\mathcal{A}_{\Pi(G)}$. Namely, we have 
\[
\mathcal{A}_{\Pi(G)}\subseteq\big\{B_\pi:\pi\in\Pi(G)\big\}\subseteq L(\mathcal{A}_{\Pi(G)})
\]
and
\begin{equation}\label{Semilattice-Partition}
 \mathcal{A}_{\tilde{\Pi}(G)}\subseteq\big\{B_\pi:\pi\in\tilde{\Pi}(G)\big\}\subseteq L(\mathcal{A}_{\tilde{\Pi}(G)}).
\end{equation}
It is worth remarking that the second inclusion may be strict in general, as the following example shows.
\begin{example}{\rm
Let $K_4$ be the complete graph with vertices $v_1,v_2,v_3,v_4$. Consider the partition selection 
\[
\mathfrak{p}=\big\{\{v_1v_2,v_3v_4\},\{v_1v_3,v_2v_4\},\{v_1v_4,v_2v_3\}\big\}.
\]
Let $\mathbb{F}$ be a finite field of characteristic $2$. Then $X=\bigcap_{\pi\in\mathfrak{p}}B_\pi=\mathbb{F}\bm 1$ is a flat of $\mathcal{A}_{\Pi(K_4)}$. Note that for any partition $\pi\in\Pi(G)$ with $k$ blocks, the corresponding boundary subspace $B_\pi$ has dimension $4-k$. Suppose that $X=B_\pi$ for some $\pi\in\Pi(K_4)$. Then $\pi$ must have three blocks, i.e., $\pi=\big\{\{v_{i_1}\},\{v_{i_2}\},\{v_{i_3},v_{i_4}\}\big\}$. It is clear that $B_\pi\ne X$. Therefore, $X$ is a flat of $\mathcal{A}_{\Pi(K_4)}$ but not of the form $B_\pi$ for any $\pi\in\Pi(K_4)$.
}
\end{example}

In order to obtain \autoref{Classification}, we give a characterization of the selections $\Pi_{\bm b}$ and $\tilde{\Pi_{\bm b}}$, based on the intersection posets $L(\mathcal{A}_{\tilde{\Pi}(G)})$ and $L(\mathcal{A}_{\Pi(G)})$.

\begin{lemma}\label{thm:equivalence condition for tilde}
Let $\bm b_1,\bm b_2\in B(G,\mathbb{F})$. Then $\bm b_1,\bm b_2\in M(\mathcal{A}_{\tilde{\Pi}(G)}/\tilde{X})$ for some flat $\tilde{X}\in L(\mathcal{A}_{\tilde{\Pi}(G)})$ if and only if 
$\tilde{\Pi}_{\bm b_1}=\tilde{\Pi}_{\bm b_2}$.
\end{lemma}
\begin{proof}
Note that if $\bm b\in M(\mathcal{A}_{\tilde{\Pi}(G)}/\tilde{X})$, then $\tilde{X}$ is the inclusion-minimal flat containing $\bm b$. The minimality of $\tilde{X}$ directly implies that 
\[
\bm b\in B_\pi\Longleftrightarrow \tilde{X}\subseteq B_\pi,\quad\forall\, \pi\in\tilde{\Pi}(G).
\]
Thus, if $\bm b_1,\bm b_2\in M(\mathcal{A}_{\tilde{\Pi}(G)}/\tilde{X})$  for some flat $\tilde{X}\in L(\mathcal{A}_{\tilde{\Pi}(G)})$, then we deduce
\[
\tilde{\Pi}_{\bm b_1}=\big\{\pi\in\tilde{\Pi}(G):\tilde{X}\subseteq B_\pi\big\}=\tilde{\Pi}_{\bm b_2}.
\]
Conversely, if $\tilde{\Pi}_{\bm b_1}=\tilde{\Pi}_{\bm b_2}$, then we have
\[
\{B_\pi\in\mathcal{A}_{\tilde{\Pi}(G)}:\bm b_1\in B_\pi\}\subseteq \{B_\pi:\pi\in\tilde{\Pi}_{\bm b_1}\}=\{B_\pi:\pi\in\tilde{\Pi}_{\bm b_2}\}\supseteq \{B_\pi\in\mathcal{A}_{\tilde{\Pi}(G)}:\bm b_2\in B_\pi\}.
\]
Together with \eqref{Flat} and \eqref{Semilattice-Partition}, we conclude that the inclusion-minimal flats $\tilde{X}_{\bm b_1}$ and $\tilde{X}_{\bm b_2}$ of $\mathcal{A}_{\tilde{\Pi}(G)}$ that contain $\bm b_1$ and $\bm b_2$ respectively, satisfy the following relation:
\begin{equation*}
\tilde{X}_{\bm b_1}=\bigcap_{B_\pi\in\mathcal{A}_{\tilde{\Pi}(G)},\,\bm b_1\in B_\pi}B_\pi=\bigcap_{\pi\in\tilde{\Pi}_{\bm b_1}}B_\pi=\bigcap_{\pi\in\tilde{\Pi}_{\bm b_2}}B_\pi=\bigcap_{B_\pi\in\mathcal{A}_{\tilde{\Pi}(G)},\,\bm b_2\in B_\pi}B_\pi=\tilde{X}_{\bm b_2}.
\end{equation*}
Consequently, $\bm b_1,\bm b_2\in M(\mathcal{A}_{\tilde{\Pi}(G)}/\tilde{X})$ with $\tilde{X}=\tilde{X}_{\bm b_1}=\tilde{X}_{\bm b _2}$. This completes the proof.
\end{proof}

By an argument similar to that in the proof of \autoref{thm:equivalence condition for tilde}, the proof of \autoref{thm:equivalence condition} follows straightforwardly, and we therefore omit its detailed proof.
\begin{lemma}\label{thm:equivalence condition}
Let $\bm b_1,\bm b_2\in B(G,\mathbb{F})$. Then $\bm b_1,\bm b_2\in M(\mathcal{A}_{\Pi(G)}/X)$  for some flat $X\in L(\mathcal{A}_{\Pi(G)})$ if and only if $\Pi_{\bm b_1}=\Pi_{\bm b_2}$.
\end{lemma}

With the above preparations, we now have enough tools to prove \autoref{Classification}.
\begin{proof}[Proof of \autoref{Classification}]
The proof of part (b) is analogous to part (a), hence we only prove part (a) and omit the proof of part (b). According to \autoref{thm:equivalence condition for tilde}, we have $\tilde{\Pi}_{\bm b_1}=\tilde{\Pi}_{\bm b_2}$, and hence $\varphi_G(\tilde{\Pi}_{\bm b_1},t)=\varphi_G(\tilde{\Pi}_{\bm b_2},t)$. It follows from \autoref{thm:nowhere-zero affine flow counting} that the second assertion in part (a) holds. It remains to prove that $L(\tilde{\mathcal{A}}_{E(G)}^{\bm b_1})\cong L(\tilde{\mathcal{A}}_{E(G)}^{\bm b_2})$.

For convenience, we assume that $i,j\in\{1,2\}$ and $\tilde{\Pi}_{\tilde{X}}=\tilde{\Pi}_{\bm b_1}=\tilde{\Pi}_{\bm b_2}$  throughout the proof. Given $S_1,S_2\subseteq E_{\rm cyc}(G)$ with $\pi(S_1^c),\pi(S_2^c)\in\tilde{\Pi}_{\tilde{X}}$, we have that $H_{S_i}^{\bm b_j}\neq\emptyset$ and 
\begin{equation}\label{eq:containments}
H^{\bm b_i}_{S_1\cup S_2}=H^{\bm b_i}_{S_1}\cap H^{\bm b_i}_{S_2}\subseteq H^{\bm b_i}_{S_j}.
\end{equation}
If $H^{\bm b_1}_{S_1}\subseteq H^{\bm b_1}_{S_2}$, then $H^{\bm b_1}_{S_1\cup S_2}=H^{\bm b_1}_{S_1}\neq\emptyset$. It follows that $\dim H^{\bm b_1}_{S_1\cup S_2}=\dim H^{\bm b_1}_{S_1}$, and $H^{\bm b_2}_{S_1\cup S_2}\neq\emptyset$ since $\tilde{\Pi}_{\bm b_1}=\tilde{\Pi}_{\bm b_2}$. Combining \eqref{eq:dimension of stratum}, we obtain $\dim H^{\bm b_2}_{S_1\cup S_2}=\dim H^{\bm b_2}_{S_1}$. Together with \eqref{eq:containments}, we deduce $H^{\bm b_2}_{S_1\cup S_2}=H^{\bm b_2}_{S_1}$, and hence $H^{\bm b_2}_{S_1}\subseteq H^{\bm b_2}_{S_2}$. Symmetrically, if $H^{\bm b_2}_{S_1}\subseteq H^{\bm b_2}_{S_2}$, then $H^{\bm b_1}_{S_1}\subseteq H^{\bm b_1}_{S_2}$. In conclusion, we have the following equivalence:
\[
H^{\bm b_1}_{S_1}\subseteq H^{\bm b_1}_{S_2}\iff H^{\bm b_2}_{S_1}\subseteq H^{\bm b_2}_{S_2}.
\]
By symmetry, the following equivalence naturally holds:
\[
H^{\bm b_1}_{S_1}=H^{\bm b_1}_{S_2}\iff H^{\bm b_2}_{S_1}=H^{\bm b_2}_{S_2}.
\]
Note that $L(\tilde{\mathcal{A}}_{E(G)}^{\bm b_i})$ consists of all affine subspaces $H_S^{\bm b_i}$ with $\pi(S^c)\in\tilde{\Pi}_{\tilde{X}}$. Consequently, a map sending $H^{\bm b_1}_S$ to $H^{\bm b_2}_S$, for $\pi(S^c)\in\tilde{\Pi}_{\tilde{X}}$, is an order-preserving bijection between $L(\tilde{\mathcal{A}}_{E(G)}^{\bm b_1})$ and $L(\tilde{\mathcal{A}}_{E(G)}^{\bm b_2})$. Therefore, $L(\tilde{\mathcal{A}}_{E(G)}^{\bm b_1})\cong L(\tilde{\mathcal{A}}_{E(G)}^{\bm b_2})$. This completes the proof.
\end{proof}

\subsection{Proof of \autoref{thm:comparison}}\label{Sec4-2}
A further goal of this section is to study the properties of coefficients of affine flow polynomials. The investigation of the unimodality and log-concavity of the coefficients of  characteristic polynomials has become a central theme in graph theory, matroid theory and hyperplane arrangement theory. Specifically, the unimodality of the sequence of unsigned coefficients of the characteristic polynomial was implicit in Rota \cite{Rota1971} and explicit in Heron \cite{Heron1972}, as a generalization of an earlier graph-theoretic conjecture of Read \cite{Read1968}. Subsequently, Welsh \cite{Welsh1976} conjectured that this sequence is log-concave. This is commonly known as the Rota-Heron-Welsh Conjecture and generalizes Hoggar's conjecture \cite{Hoggar1974} for chromatic polynomials. These conjectures have been confirmed by Huh et al. \cite{AHE2018,Huh2012,Huh-Katz2012}. As a direct result, we conclude the following corollary by \autoref{thm:nowhere-zero affine flow counting} and \autoref{thm:nowhere-zero affine flow counting1}.
\begin{corollary}
Let $\bm b\in B(G,\mathbb{F})$. Then,  the coefficients of the affine flow polynomial $\varphi_G(\tilde{\Pi}_{\bm b},t)$  are nonzero and alternate in sign, and the sequence of $w_i(\tilde{\Pi}_{\bm b})$ is unimodal and log-concave.
\end{corollary}

Most recently, Fu, Ren and Wang provided a combinatorial description for the unsigned coefficients of admissible assigning polynomials $\varphi_G(\alpha,t)$ in \cite[Theorem 1.5]{FRW2025} by introducing $\bm b$-compatible broken bonds. Subsequently, they further used this combinatorial interpretation to establish a unified order-preserving relation from $\{0,1\}$-assignings to assigning polynomials when both are naturally ordered. Motivated by their research, we provide an analogous comparison relation for the coefficients of affine flow polynomials in \autoref{thm:comparison}, based on the reduced boundary arrangement and the boundary arrangement. To show it, we need Whitney's celebrated Broken Circuit Theorem \cite{Whitney1932}, which is an important tool for computing the unsigned coefficients of characteristic polynomials. 

Let $\mathcal{A}=\{H_1,\ldots,H_m\}$ be a hyperplane arrangement in the $d$-dimensional vector space. A subset $S\subseteq[m]$ is {\em affinely independent} if  $\bigcap_{i\in S}H_i\ne\emptyset$ and ${\rm codim} \bigcap_{i\in S}H_i=|S|$. Similarly, $S$ is said to be {\em affinely dependent} if $\bigcap_{i\in S}H_i\ne\emptyset$ and ${\rm codim} \bigcap_{i\in S}H_i<|S|$. A minimal affinely dependent subset $S$ of $[m]$ is referred to as an {\em affine circuit} with respect to $\mathcal{A}$. Alternatively, $S$ is an affine circuit if and only if
\[
\bigcap_{i\in S}H_i\ne\emptyset\quad\And\quad {\rm codim} \bigcap_{i\in S}H_i={\rm codim} \bigcap_{i\in S\smallsetminus\{j\}}H_i=|S|-1,\quad\forall\, j\in S.
\]
Given a total order $\prec$ on $[m]$, an {\em affine broken circuit} is a subset of $[m]$ obtained from an affine circuit by deleting  the minimal element. For more information on affine broken circuits of hyperplane arrangements, we refer the reader to \cite{Forge-Zaslavsky2016,Orlik-Terao1992}.
\begin{theorem}[Affine NBC Theorem \cite{Orlik-Terao1992}]\label{affine-NBC} Let $\mathcal{A}=\{H_1,\ldots, H_m\}$ be a hyperplane arrangement in a $d$-dimensional vector space $V$. Write $\chi(\mathcal{A},t)$ as $\chi(\mathcal{A},t)=\sum_{i=0}^d(-1)^iw_i(\mathcal{A})t^{d-i}$. Then, every unsigned coefficient $w_i(\mathcal{A})$ equals the number of  affinely independent subsets of $[m]$ of size $i$ that contain no affine broken circuits.
\end{theorem}

We also need the following lemma, which shows that the collections of affinely independent subsets for different restricted arrangements $\mathcal{A}_{E(G)}^{\bm b}$ are identical.
\begin{lemma}\label{Independent}
Let $\bm b_1,\bm b_2\in B(G,\mathbb{F})$ and $S\subseteq E(G)$. If $H_S^{\bm b_1}\ne\emptyset$ and ${\rm codim}H_S^{\bm b_1}=|S|$, then $H_S^{\bm b_2}\ne\emptyset$ and ${\rm codim}H_S^{\bm b_2}=|S|$.
\end{lemma}
\begin{proof}
From \eqref{eq:dimension of stratum}, we have $\dim H_S^{\bm b_1}=n(S^c)$. Since ${\rm codim}H_S^{\bm b_1}=\dim F(G,\bm b_1;\mathbb{F})-\dim H_S^{\bm b_1}=|S|$, we deduce that $n(G)-n(S^c)=|S|$. Together with $n(S^c)=|E(G\setminus S)|-|V(G)|+c(G\setminus S)$, we further derive $c(G\setminus S)=c(G)$. Consequently, $|\pi(S^c)|=|\pi(E(G))|$. As $\pi(S^c)$ is a refinement of $\pi(E(G))$, we conclude that $\pi(S^c)=\pi(E(G))$. It follows that $\bm b_2\in B(G,\mathbb{F})=B_{\pi(E(G))}=B_{\pi(S^c)}$. So, we arrive at $H_S^{\bm b_2}\ne\emptyset$ by \eqref{eq:equivalence for coset intersection nonempty}. Applying \eqref{eq:dimension of stratum} again, we obtain $\dim H_S^{\bm b_2}=n(S^c)$, and hence ${\rm codim}H_S^{\bm b_2}=|S|$. This completes the proof.
\end{proof}

With the above preparations, we now proceed to prove \autoref{thm:comparison}.
\begin{proof}[Proof of \autoref{thm:comparison}]
The proof of part (b) is analogous to part (a), hence we only prove part (a) and omit the proof of part (b). For each $i=1,2$, let ${\rm NBC}(\tilde{\mathcal{A}}_{E(G)}^{\bm b_i})$ denote the set of affinely independent subsets of $\tilde{\mathcal{A}}_{E(G)}^{\bm b_i}$ that do not contain affine broken circuits. According to part (b) of \autoref{thm:nowhere-zero affine flow counting} and \autoref{affine-NBC}, it suffices to verify that 
\[
{\rm NBC}(\tilde{\mathcal{A}}_{E(G)}^{\bm b_1})\subseteq {\rm NBC}(\tilde{\mathcal{A}}_{E(G)}^{\bm b_2}).
\]
We prove this by contradiction. Suppose there exists an element $S_0\in {\rm NBC}(\tilde{\mathcal{A}}_{E(G)}^{\bm b_1})$ such that $S_0\notin{\rm NBC}(\tilde{\mathcal{A}}_{E(G)}^{\bm b_2})$. Note from \autoref{Independent} that for any $S\subseteq E_{\rm cyc}(G)$, $S$ is affinely independent with respect to $\tilde{\mathcal{A}}_{E(G)}^{\bm b_1}$ if and only if $S$ is affinely independent with respect to $\tilde{\mathcal{A}}_{E(G)}^{\bm b_2}$. Thus, $S_0$ is affinely independent with respect to $\tilde{\mathcal{A}}_{E(G)}^{\bm b_i}$ with $i=1,2$. Since $S_0\notin{\rm NBC}(\tilde{\mathcal{A}}_{E(G)}^{\bm b_2})$, there exists an element $e_0\in E_{\rm cyc}(G)\smallsetminus S_0$ such that $S_0\sqcup e_0$ contains an affine circuit $C_0$ with $e_0\in C_0$, and $C_0\smallsetminus \{e_0\}\subseteq S_0$ is an affine broken circuit of $\tilde{\mathcal{A}}_{E(G)}^{\bm b_2}$. Consequently, $H_{C_0}^{\bm b_2}\ne\emptyset$ and ${\rm codim}H_{C_0}^{\bm b_2}=|C_0|-1$. Hence, $\bm b_2\in B_{\pi(C_0^c)}$ via \eqref{eq:equivalence for coset intersection nonempty}. As $X_1\subseteq X_2$, every hyperplane $B_\pi\in\mathcal{A}_{\tilde{\Pi}(G)}$ containing $X_2$ necessarily contains $X_1$. It follows that  $\bm b_1\in B_{\pi(C_0^c)}$. Applying  \eqref{eq:equivalence for coset intersection nonempty} again, we obtain $H_{C_0}^{\bm b_1}\ne\emptyset$. According to \autoref{Independent}, $C_0$ is also an affine circuit of $\tilde{\mathcal{A}}_{E(G)}^{\bm b_1}$. Consequently, $C_0\smallsetminus \{e_0\}\subseteq S_0$ forms an affine broken circuit of $\tilde{\mathcal{A}}_{E(G)}^{\bm b_1}$, which contradicts the assumption $S_0\in {\rm NBC}(\tilde{\mathcal{A}}_{E(G)}^{\bm b_1})$. Therefore, the desired inclusion relation holds. 
\end{proof}
\section{Proof of \autoref{Dec1}}\label{Sec5}
This section is devoted to proving \autoref{Dec1}. To this end, we first introduce valuation theory on characteristic polynomials. Roughly speaking, the characteristic polynomial of a hyperplane arrangement measures the size of its complement. Let $V$ be a vector space over an infinite field $\mathbb{F}$, and $B(V)$ be the Boolean algebra generated by all affine subspaces of $V$.  A {\em valuation} $\nu$ on $B(V)$ is a set function from $B(V)$ to a commutative ring $R$ such that
\[
\nu(\emptyset)=0\quad\And\quad \nu(X)+\nu(Y)=\nu(X\cap Y)+\nu(X\cup Y), \For X,Y,X\cup Y\in B(V).
\]
Ehrenborg and Readdy in 1998 obtained a unique valuation (translation-invariant) $\nu: B(V)\to \mathbb{Z}[t]$ satisfying $\nu(W)=t^{\dim (W)}$ for any nonempty affine subspace $W$ of $V$ in \cite[Theorem 2.1]{Ehrenborg-Readdy1998},  which we call the {\em characteristic valuation} on $V$. They further showed that the characteristic polynomial of  a hyperplane arrangement $\mathcal{A}$ can be viewed as a valuation of its complement $M(\mathcal{A})$. More details on the valuation theory of characteristic polynomials can be found in \cite{Bjorner1997, Chen2000}.
\begin{theorem}{\rm\cite[Theorem 2.3]{Ehrenborg-Readdy1998}}\label{thm:valuation and characteristic polynomial}
Let $\mathcal{A}$ be a subspace arrangement in a vector space $V$ over an infinite field $\mathbb{F}$, and $\nu$ be the characteristic valuation on $V$. Then the characteristic polynomial $\chi(\mathcal{A},t)$ is given by
\[
\chi(\mathcal{A},t)=\nu\big(M(\mathcal{A})\big).
\]
\end{theorem}

For any subset $X$ of $V$ over an infinite field $\mathbb{F}$,  let $1_X$ denote the indicator function of the set $X$, i.e., $1_X(x)=1$ if $x\in X$, and $1_X=0$ otherwise. A function $f:V\rightarrow\mathbb{Z}[t]$ is said to be {\em simple} if $f$ can be written as a linear combination of indicator functions of sets from the Boolean algebra $B(V)$, that is,  $f=\sum_{i=1}^l c_i 1_{X_i}$ for some $c_i\in\mathbb{Z}[t]$ and  $X_i\in B(V)$.  Then the characteristic valuation $\nu$ defines an integral
\[
\int f{\rm d}\nu := \sum_{i=1}^l c_i \nu({X_i})
\]
for each simple function $f=\sum_{i=1}^l c_i 1_{X_i}$, where $X_i\in B(V)$. The following proposition is a ``Fubini-type" theorem.
\begin{proposition}[\cite{Ehrenborg-Readdy1998}, Proposition 2.2]\label{thm:Fubini-type theorem}
Let $\nu_1,\nu_2$ and $\nu$ be the characteristic valuations on vector spaces $V_1,V_2$ and $V_1\times V_2$, respectively. Let $f(x,y)$ be a
simple function on $V_1\times V_2$. Then for each $x\in V_1$, $f_x(y):=f(x,y)$ is a simple function on $V_2$, and $\int f_x(y){\rm d}\nu_2$
is a simple function in terms of $x$ on $V_1$. Moreover,
\[
\int f(x,y){\rm d}\nu=\int\int f_x(y){\rm d}\nu_2{\rm d}\nu_1.
\]
\end{proposition}

By applying valuation theory, we give a proof of \autoref{Dec1} below.
\begin{proof}[Proof of part (a) of \autoref{Dec1}]
Let $V_1=\mathbb{F}^{E(G)}$ and $V_0=\mathbb{F}^{V(G)}$. Consider the graph $G(\partial)$ of the boundary operator, which is a subspace of $V_1\times V_0$ consisting of pairs $(\bm c,\bm b)$ such that $\partial\bm c=\bm b$. Namely,
\[
G(\partial):=\big\{(\bm c,\partial\bm c):\bm c\in V_1\big\}\subseteq V_1\times V_0.
\]
For each edge $e\in E(G)$, we denote by $G_e(\partial)$ the hyperplane of $G(\partial)$ consisting of pairs $(\bm c,\bm b)$ such that $\partial\bm c=\bm b$ with $\bm c(e)=0$. The set $\mathcal{A}_\partial:=\big\{G_e(\partial):e\in E(G)\big\}$ forms a hyperplane arrangement of $G(\partial)$. It is obvious that the complement $M(\mathcal{A}_\partial)=G(\partial)\smallsetminus\bigcup_{e\in E(G)}G_e(\partial)$ of $\mathcal{A}_\partial$ is an element of $\mathcal{B}(V_1\times V_0)$, and hence its indicator function $f:=1_{M(\mathcal{A}_\partial)}$ is a $\mathbb{Z}[t]$-valued simple function on $V_1\times V_0$. For any $\bm b\in V_0$, the function $f_{\bm b}:=f(\cdot,\bm b)$ is a simple function on $V_1$ such that $f_{\bm b}=1_{M(\mathcal{A}_{E(G)}^{\bm b})}$ if $\bm b\in B(G,\mathbb{F})$, and $f_{\bm b}=0$ otherwise. Note from \eqref{DA} that the boundary space has the set partition:  $B(G,\mathbb{F})=\bigsqcup_{X\in L(\mathcal{A}_{\Pi(G)})}M(\mathcal{A}_{\Pi(G)}/X)$. This implies that for each $\bm b\in B(G,\mathbb{F})$, there is a unique flat $X_{\bm b}$ of $\mathcal{A}_{\Pi(G)}$ such that $\bm b\in M(\mathcal{A}_{\Pi(G)}/X_{\bm b})$. According to \autoref{thm:valuation and characteristic polynomial} and part (b) in \autoref{thm:nowhere-zero affine flow counting1}, for fixed $\bm b\in B(G,\mathbb{F})$, we have
\[
\int f_{\bm b}\,d\nu_1=\nu\big(M(\mathcal{A}_{E(G)}^{\bm b})\big)=\chi(\mathcal{A}_{E(G)}^{\bm b},t)=\varphi(\Pi_{\bm b},t)=\varphi_G(X_{\bm b},t).
\]
Consequently, $\int f_{\bm b}\,d\nu_1$ is constant on each complement  $M(\mathcal{A}_{\Pi(G)}/X)$ and vanishes outside $B(G,\mathbb{F})$. Thus, $\int f_{\bm b}\,d\nu_1$ can be written as
\[
\int f_{\bm b}\,d\nu_1=\sum_{X\in L(\mathcal{A}_{\Pi(G)})}\varphi_G(X,t)1_{M(\mathcal{A}_{\Pi(G)}/X)}.
\]
Integrating over $V_0$, we further obtain
\begin{align*}
\int\!\int f_{\bm b}\,d\nu_1d\nu_0
&=\sum_{X\in L(\mathcal{A}_{\Pi(G)})}\varphi_G(X,t)\,\nu_0\big(M(\mathcal{A}_{\Pi(G)}/X)\big)\notag\\
&=\sum_{X\in L(\mathcal{A}_{\Pi(G)})} \varphi_G(X,t)\,\chi(\mathcal{A}_{\Pi(G)}/X,t).
\end{align*}
On the other hand, applying \autoref{thm:valuation and characteristic polynomial} directly yields
\[
\int 1_{M(\mathcal{A}_\partial)}\,d\nu=\nu\big(M(\mathcal{A}_\partial)\big)=\chi(\mathcal{A}_\partial,t).
\]
The intersection poset of the hyperplane arrangement $\mathcal{A}_\partial$ is isomorphic to the Boolean lattice $(2^{|E(G)|},\subseteq)$ . Therefore, its characteristic polynomial $\chi(\mathcal{A}_\partial,t)$ equals $(t-1)^{|E(G)|}$. By \autoref{thm:Fubini-type theorem}, we conclude that
\[
(t-1)^{|E(G)|}=\sum_{X\in L(\mathcal{A}_{\Pi(G)})}\varphi_G(X,t)\,\chi(\mathcal{A}_{\Pi(G)}/X,t).
\]

When $\mathbb{F}$ is a finite field of $q$ elements, the valuations $\nu _0$, $\nu_1$, and $\nu$ are understood as
counting measures, and the variable $t$ is replaced by the number $q$. 
\end{proof}

To verify part (b) of \autoref{Dec1}, we now consider another hyperplane arrangement $\mathcal{A}^{\rm cyc}_\partial$ of $G(\partial)$ consisting of the hyperplanes $G_e(\partial)$ indexed by cycle edges, i.e.,
\[
\mathcal{A}^{\rm cyc}_\partial:=\big\{G_e(\partial):e\in E_{\rm cyc}(G)\big\}.
\]
Similarly, the complement $M(\mathcal{A}^{\rm cyc}_\partial)=G(\partial)\smallsetminus\bigcup_{e\in E_{\rm cyc}(G)}G_e(\partial)$ of $\mathcal{A}^{\rm cyc}_\partial$ is also an element of $\mathcal{B}(V_1\times V_0)$, and hence its indicator function $f:=1_{M(\mathcal{A}^{\rm cyc}_\partial)}$ is a $\mathbb{Z}[t]$-valued simple function on $V_1\times V_0$. The boundary space also decomposes as $B(G,\mathbb{F})=\bigsqcup_{{\tilde X}\in L(\mathcal{A}_{{\tilde\Pi}(G)})}M(\mathcal{A}_{{\tilde\Pi(G)}}/{\tilde X})$. Thus, we replace $\mathcal{A}_\partial$ and $\mathcal{A}_{\Pi(G)}$ in the proof of part (a) of \autoref{Dec1} with $\mathcal{A}^{\rm cyc}_\partial$ and $\mathcal{A}_{{\tilde\Pi(G)}}$, respectively. Analogous to part (a) of \autoref {Dec1}, we can derive the following relation:
\[
t^{|E_{\rm cut}(G)|}(t-1)^{|E_{\rm cyc}(G)|}=\int 1_{M(\mathcal{A}^{\rm cyc}_\partial)}d\nu=\int\!\int f_{\bm b}\,d\nu_1d\nu_0=\sum_{\tilde{X}\in L(\mathcal{A}_{\tilde{\Pi}(G)})}\tilde{\varphi}_G(\tilde{X},t)\,\chi(\mathcal{A}_{\tilde{\Pi}(G)}/\tilde{X},t).
\]
Likewise, when $\mathbb{F}$ is a finite field of $q$ elements, the valuations $\nu _0$, $\nu_1$, and $\nu$ are understood as
counting measures, and the variable $t$ is replaced by the number $q$. By summarizing the above arguments, we conclude that part (b) of \autoref{Dec1} holds.

\section*{Acknowledgements}
This work is supported by the Guangdong Basic and Applied Basic Research Foundation (Grant No. 2026A1515012237, Grant No. 2026A1515012543). 


\begin{thebibliography}{99}\setlength{\itemsep}{-.mm}
\bibitem{AHE2018}
K. Adiprasito, J. Huh, E. Katz, Hodge theory for combinatorial geometries, Ann. of Math. (2) 188 (2018), 381--452.

\bibitem{Birkhoff1912}
G. D. Birkhoff, A determinant formula for the number of ways of coloring a map, Ann. Math. 14 (1912), 42--46.

\bibitem{Bjorner1997}
A. Bj\"{o}rner, T. Ekdahl, Subspace arrangements over finite fields: cohomological and enumerative aspects, Adv. Math. 129 (1997), 159--187.

\bibitem{Bondy1976}
J. A. Bondy, U. S. R. Murty, Graph Theory with Applications, Macmillan, London, 1976.

\bibitem{Chen2000}
B. Chen, On characteristic polynomials of subspaces arrangements, J. Combin. Theory, Ser. A 90 (2000), 347--352.

\bibitem{Ehrenborg-Readdy1998}
R. Ehrenborg, M. A. Readdy, On valuations, the characteristic polynomial, and complex subspace arrangements. Adv. Math. 134 (1998), 32--42.

\bibitem{Forge-Zaslavsky2016}
D. Forge, T. Zaslavsky, Lattice points in orthotopes and a huge polynomial Tutte invariant of weighted gain graphs, J. Combin. Theory Ser. B 118 (2016), 186--227.

\bibitem{FRW2025}
H. Fu, X. Ren, S. Wang, Counting flows of b-compatible graphs, Adv. in Appl. Math. 168 (2025),  Paper No. 102901.

\bibitem{Heron1972}
A. P. Heron, Matroid polynomials, In: Combinatorics, Institute of Math. and its Applications, Southend-on-Sea, 1972, pp. 164--202.

\bibitem{Hoggar1974}
S. G. Hoggar, Chromatic polynomials and logarithmic concavity, J. Combin. Theory Ser. B 16 (1974), 248--254.

\bibitem{Huh2012}
J. Huh, Milnor numbers of projective hypersurfaces and the chromatic polynomial of graphs, J. Amer. Math. Soc. 25 (2012), 907--927.

\bibitem{Huh-Katz2012}
J. Huh, E. Katz, Log-concavity of characteristic polynomials and the Bergman fan of matroids, Math. Ann. 354 (2012), 1103--1116.

\bibitem{Jaeger1988}
F. Jaeger, Nowhere-zero flow problems, in: L. W. Beineke, R. J. Wilson (Eds.), Selected Topics in Graph Theory, Vol. 3, Academic Press, New York, 1988, pp. 71--95.

\bibitem{Jaeger1992}
F. Jaeger, N. Linial, C. Payan, M. Tarsi, Group connectivity of graphs--a nonhomongenous analogue of nowhere-zero flow properties,  J. Combin. Theory Ser. B 56 (1992), 165--182.

\bibitem{Kochol2002}
M. Kochol, Tension polynomials of graphs, J. Graph Theory 40 (2002), 137--146.

\bibitem{Kochol2022}
M. Kochol, Polynomials counting nowhere-zero chains in graphs, Electron. J. Combin. 29 (2022) P1.19.

\bibitem{Kochol2024}
M. Kochol. Polynomials counting nowhere-zero chains associated with homomorphisms. Mathematics 12 (2024), 3218.

\bibitem{Lai2006}
H.-J. Lai, X. Yao, Group connectivity of graphs with diameter at most 2, European J. Combin. 27 (2006), 436--447.

\bibitem{Orlik-Terao1992}
P. Orlik, H. Terao, Arrangements of Hyperplanes, Springer-Verlag, Berlin, 1992.

\bibitem{Read1968}
R. C. Read, An introduction to chromatic polynomials, J. Combin. Theory 4 (1968), 52--71.

\bibitem{Rota1971}
G.-C. Rota, Combinatorial theory, old and new, In: Actes du Congr\`{e}s International des Math\'{e}maticiens. Tome 3 (Nice, 1970), Gauthier-Villars, Paris, 1971, pp. 229--233.

\bibitem{Seymour1995}
P. D. Seymour, Nowhere-zero flows, Appendix to Chapter 4, in: R.L. Graham, M. Grötschel, L. Lovász (Eds.), Handbook of Combinatorics, vol. 1, North-Holland (Elsevier), Amsterdam, 1995, pp. 289--299.

\bibitem{Stanley2007}
R. P. Stanley, An introduction to hyperplane arrangements, In: E. Miller, V. Reiner, B. Sturmfels (Eds.), Geometric Combinatorics, IAS/Park City Math. Ser. vol. 13, Amer. Math. Soc. Providence, RI, 2007, pp. 389--496.

\bibitem{Tutte1949}
W. T. Tutte, On the imbedding of linear graphs in surfaces, Proc. London Math. Soc. (2) 51 (1949), 474--483.
 
\bibitem{Tutte1954}
W. T. Tutte, A contribution to the theory of chromatic polynomials, Canad. J. Math. 6 (1954) 89--91.

\bibitem{Tutte1967}
W. T. Tutte, On dischromatic polynomials, J. Combin. Theory 2 (1967), 301--320.

\bibitem{Welsh1976}
D. J. A. Welsh, Matroid Theory, Academic Press, London (1976). Reprinted (2010), Dover, Mineola.

\bibitem{Whitney1932}
H. Whitney, A logical expansion in mathematics, Bull. Amer. Math. Soc. 38 (1932), 572--579.

\bibitem{Whitney1932-1}
H. Whitney, The coloring of graphs, Ann. Math. 33(2) (1932), 688--718.
\end{thebibliography}
\end{document}